\documentclass[a4paper]{article}
\usepackage{amsmath}
\usepackage{amssymb}
\usepackage{amsthm}
\usepackage{mathrsfs}
\newcommand{\si}{\sigma}
\newcommand{\la}{\lambda}
\newcommand{\ol}{\overline}

\newcommand{\pa}{\partial}
\newcommand{\La}{\Lambda}
\newcommand{\al}{\alpha}
\newcommand{\be}{\beta}
\newcommand{\de}{\delta}
\newcommand{\ga}{\gamma}
\newcommand{\ka}{\kappa}

\newcommand{\om}{\omega}
\newcommand{\ve}{\varepsilon}
\newcommand{\cd}{\cdot}

\newcommand{\R}{{\mathbb R}}

\newcommand{\N}{{\mathbb N}}
\newcommand{\cE}{{\cal E}}

\newcommand{\cB}{{\cal B}}

\newcommand{\cM}{{\cal M}}
\newcommand{\cN}{{\cal N}}
\newcommand{\cH}{{\cal H}}

\newcommand{\tM}{\widetilde{M}}
\newcommand{\tcM}{\widetilde{\cM}}
\newcommand{\tN}{\widetilde{N}}
\newcommand{\tcN}{\widetilde{\cN}}
\newcommand{\mM}{{\mathfrak M}}

\renewcommand{\(}{\left(}
\renewcommand{\)}{\right)}
\newcommand{\Th}{\Theta}
\renewcommand{\i}{{\rm i}}
\renewcommand{\th}{\theta}

\newtheorem{Thm}{\indent\bf Theorem}[section]
\newtheorem{Lemm}{\indent\bf Lemma}[section]
\newtheorem{Prop}{\indent\bf Proposition}[section]

\theoremstyle{remark}
  \newtheorem{Rem}{\indent\sc Remark}[section]
\theoremstyle{definition}
  \newtheorem{Def}{\indent Definition}[section]
  \newtheorem{Ex}{\indent\sc Example}[section]
\numberwithin{equation}{section}

\begin{document}
\title{A class of non-analytic functions for the global solvability of Kirchhoff equation}
\author{Fumihiko Hirosawa
\\
{\small
Department of Mathematical Sciences, Yamaguchi University, 
Yamaguchi 753-8512, Japan}
\\
{\small E-mail: hirosawa@yamaguchi-u.ac.jp}}
\date{}

\maketitle

\begin{abstract}
We consider the global solvability to the Cauchy problem of Kirchhoff equation with generalized classes of Manfrin's class. 
Manfrin's class is a subclass of Sobolev space, but we shall extend this class as a subclass of the ultradifferentiable class, and we succeed to prove the global solvability of Kirchhoff equation with large data in wider classes from the previous works. 
\end{abstract}

\begin{center}
\end{center}

\begin{center}
\end{center}

\section{Introduction}
We consider the global solvability to the Cauchy problem of Kirchhoff equation:
\begin{equation}\label{K}
\begin{cases}
\displaystyle{
  \pa_t^2 u(t,x)-\(1+\|\nabla u(t,\cd)\|^2\)\Delta u(t,x)=0,
  \quad
  (t,x)\in(0,\infty)\times \R^n,}
\\
  u(0,x)=u_0(x),
  \quad
  (\pa_t u)(0,x)=u_1(x),
  \quad
  x\in \R^n,
\end{cases}
\end{equation}
where 
$\nabla=(\pa_{x_1},\cdots,\pa_{x_n})$, 
$\Delta=\sum_{k=1}^n\pa_{x_k}^2$ and 
$\|\cd\|$ denotes the usual $L^2(\R^n)$ norm. 
The original Kirchhoff equation was introduced by Kirchhoff \cite{K} to describe the transversal oscillation of a stretched string, which corresponds to the equation of (\ref{K}) with $n=1$ and $\|\nabla u(t,\cd)\|^2$ is replaced to $\int_I |\pa_x u(t,x)|^2\,dx$ with a finite interval $I$ of $\R$. 
The global solvability in realanalytic class and the local solvability in appropriate Sobolev spaces were proved in Bernstein \cite{B}, and the global solvability with small data was proved in Greenberg-Hu \cite{GH}. 
After that, many people considered the global solvability of the Kirchhoff equations, for instance \cite{AS}, \cite{DS}, \cite{Me}, \cite{Po}, \cite{Yg}, \cite{Yz}, etc; see also \cite{S} more details of Kirchhoff equation. 
However, the most basic problems; the global solvability for non-realanalytic initial data and the non-existence of the global solution have not solved yet. 
Here we note that the difference between the Cauchy problem and mixed problem are not essential for the proof of the global solvability, thus we did not distinguish them for the introduction of the previous works. 
The class of non-analytic data for the global solvability was studied from two different points of view: 
the first one is a quasianalytic class ${\cal Q}_{L^2}$ by Nishihara \cite{Ni} (see also \cite{GG,H98}), and the second one is Manfrin's class introduce in \cite{M02, M05} and \cite{H06}; 
both classes will be introduced below. 
Briefly, we consider the global solvability of (\ref{K}) in the classes, 
which are extensions of Manfrin's class.

Let $u(t,x)$ be a solution to (\ref{K}). 
We define $\cE(t,\xi)$ by 
\begin{equation*}
  \cE(t,\xi):=\frac12\(|\xi|^2|\hat{u}(t,\xi)|^2+|\pa_t \hat{u}(t,\xi)|^2\), 
\end{equation*}
where $\hat{u}$ denotes the partial Fourier transformation of $u(t,x)$ with respect to $x\in \R^n$. 
In particular, $\cE(0,\xi)=\cE(0,\xi;u_0,u_1)$ denotes 
\begin{equation*}
  \cE(0,\xi;u_0,u_1):
  =\frac12\(|\xi|^2|\hat{u}_0(\xi)|^2+|\hat{u}_1(\xi)|^2\).  
\end{equation*}
Moreover, li
for an integer $m$ satisfying $m\ge2$, real numbers $\rho,\eta$ satisfying $\rho\ge1$, $\eta>0$ we define 
$G_m(\rho,\eta)=G_m(\rho,\eta;u_0,u_1)$ by 
\begin{equation*}
  G_m(\rho,\eta;u_0,u_1):=\int_{|\xi|\ge \rho}
  \(\frac{|\xi|}{\rho}\)^m \exp\(\frac{\eta|\xi|}{\(\frac{|\xi|}{\rho}\)^{m}}\)
  \cE(0,\xi;u_0,u_1)\,d\xi. 
\end{equation*}
Let ${\cal L}$ be the set of all strictly increasing sequence of positive real numbers $\{\rho_j\}_{j=1}^\infty$ satisfying $\rho_1\ge 1$ and $\lim_{j\to\infty}\rho_j=\infty$. 
Then Manfrin's class is defined as follows: 
\begin{Def}[Manfrin's class]
For $m\ge 2$ we define Manfrin's class $B_\Delta^{(m)}$ by
\begin{equation*}
  B^{(m)}_\Delta:=\bigcup_{\eta>0}\left\{(u_0,u_1)\in H^{\frac{m}{2}+1}\times H^{\frac{m}{2}}\:;\:
  \exists\{\rho_j\}\in {\cal L},\;
  \sup_j \left\{ G_m(\rho_j,\eta;u_0,u_1)\right\}<\infty
  \right\}, 
\end{equation*}
where $H^j$ denote the usual Sobolev class of order $j$. 
\end{Def}
Then we have the following theorem for the global solvability of (\ref{K}): 
%
\begin{Thm}[\cite{H06,M02,M05}]\label{Thm0}
Let $m\ge 2$. 
If $(u_0,u_1) \in B^{(m)}_\Delta$, 
then (\ref{K}) has a unique strong solution satisfying 
$(u(t,x),u_t(t,x)) \in H^{\frac{m}{2}+1}\times H^{\frac{m}{2}}$ 
for any $t\in (0,\infty)$. 
Here we say that $u(t,x)$ is a strong solution if 
$u(t,x) \in \bigcap^2_{k=0} C^k([0,T);H^{2-k})$. 
\end{Thm}
\begin{Rem}
Theorem \ref{Thm0} was proved for $m=2$, $m=3$, and $m\ge 4$ in \cite{M02}, \cite{M05}, and \cite{H06} respectively. 
\end{Rem}

For Manfrin's class $B^{(m)}_\Delta$ we have the following properties: 
\begin{Lemm}[\cite{H06,M02,M05}]\label{B_Delta^m-prop}
Let $\cH(\{k!\})$, and ${\cal Q}_{L^2}$ be the realanalytic class, and the quasianalytic class respectively, 
which will be given in Definition \ref{udclass}, Lemma \ref{lemm-uldclass} (iii) and (iv). 
For any $m\ge 2$, the following properties are established$:$ 
\begin{itemize}
\item[(i)]
There exists $(u_0,u_1)\in B^{(m)}_\Delta$ such that 
$u_1\not\in \cH(\{k!\})$. 
\item[(ii)]
There exists $(u_0,u_1)\in B^{(m)}_\Delta$ such that $u_1\not\in {\cal Q}_{L^2}$.  
On the other hand, there exists $u_1 \in {\cal Q}_{L^2}$ such that 
$(0,u_1)\not\in B^{(m)}_\Delta$. 
\item[(iii)]
For any $\ve>0$, there exists $(u_0,u_1)\in B^{(m)}_\Delta$ such that 
$(u_0,u_1)\not\in H^{\frac{m}{2}+1+\ve}\times H^{\frac{m}{2}+\ve}$.  
\item[(iv)]
There exists $(u_0,u_1)\in B^{(m+1)}_\Delta$ such that $(u_0,u_1)\not\in B^{(m)}_\Delta$.  
\end{itemize} 
\end{Lemm}

\begin{Rem}
Denoting $p(r)=p(r;\rho):=(r/\rho)^m$ and 
$q(r)=q(r;\rho):=\exp(\eta r/(r/\rho)^{m})$ for $r\ge \rho$, 
the weight function $p(|\xi|)q(|\xi|)$ of $G_m(\rho,\eta)$ to $\cE(0,\xi)$ is not a standard one, because $p(r)q(r)$ is not monotone increasing with respect to $r$. 
Moreover, we observe the followings$:$ 
\begin{itemize}
\item 
$p(r;\rho)q(r;\rho)$ is monotone decreasing on 
$[\rho, \rho(\eta\rho(m-1)/m)^{1/(m-1)}]$, 
and monotone increasing on $[\rho(\eta\rho(m-1)/m)^{1/(m-1)},\infty)$ 
with respect to $r$. 
Thus the weight function $p(|\xi|;\rho)q(|\xi|;\rho)$ is non-standard in 
\begin{equation*}
  \Omega(\rho):=\left\{
  \xi\in \R^n\:;\: \rho \le |\xi| \le 
  \rho\(\frac{\eta\rho(m-1)}{m}\)^{\frac{1}{m-1}}
  \right\}.
\end{equation*}
\item
Noting $\lim_{r\to\infty}q(r;\rho)=1$, 
$p(|\xi|;\rho)q(|\xi|;\rho)$ is the usual weight of Sobolev type outside of $\Omega(\rho)$. 
\item
$q(|\xi|;\rho)$ has influence only near to $|\xi|=\rho$, and the effective area is smaller as $m$ larger. 
However, the influence of $q(|\xi|;\rho)$ cannot be ignored since $\rho \to \infty$. 
\end{itemize}
\end{Rem}

For strictly increasing positive continuous functions $\cM$ and $\tcM$, we consider a generalization of $G_m(\rho,\eta)$ as follows: 
\begin{equation*}
  G(\rho,\eta,\cM,\tcM):=\int_{|\xi|\ge \rho}
  \tcM\(\frac{|\xi|}{\rho}\)
  \exp\(\frac{\eta|\xi|}{\cM\(\frac{|\xi|}{\rho}\)}\)
  \cE(0,\xi)\,d\xi. 
\end{equation*}
Indeed, we see that 
$G(\rho,\eta,\cM,\tcM)=G_m(\rho,\eta)$ for $\cM(r)=\tcM(r)=r^m$. 
The main purpose of the present paper is to extend the order of 
$\cM(r),\widetilde{\cM}(r)$ from polynomial to infinite order.

\section{Preliminaries and main theorem}
\subsection{Log convex sequences}
Let us introduce {\it log convex sequences} and some properties of them. 
\begin{Def}[Log convex sequence]
Let $\{M_k\}=\{M_k\}_{k=0}^\infty$ be a sequence of positive real numbers. 
We call that $\{M_k\}$ is a {\it log convex sequence} if $0<M_0\le M_1$ and the following inequalities hold for any $k\in \N:$
\begin{equation}\label{lc}
  \frac{M_k}{k M_{k-1}} \le \frac{M_{k+1}}{(k+1)M_k}. 
\end{equation}
\end{Def}
\begin{Ex}\label{lc-ex}
The following sequences $\{M_k\}$ are log convex: 
\begin{itemize}
\item[(i)]
$M_k=\si^k k!^\nu$ with $\si>0$ and $\nu\ge 1$. 
\item[(ii)]
$M_0=1$ and $M_k=\prod_{j=1}^k j\exp(j^s)$ for $k\ge 1$ with $s>0$. 
\end{itemize}
\end{Ex}
\begin{Rem}
The usual definition of log convex sequence is not (\ref{lc}) but 
$M_k/M_{k-1} \le M_{k+1}/M_k$. In this case, $\{\si^k k!^\nu\}$ with $0\le \nu<1$ is also log convex. 
\end{Rem}
From now on, we suppose that the log convex sequences $\{M_k\}$ satisfy $M_0=M_1=1$ without loss of generality. 

For a log convex sequence we define the {\it associated function} as follows: 
\begin{Def}
For a log convex sequence $\{M_k\}$, we define the {\it associated function} $\mM(r;\{M_k\})$ on $(0,\infty)$ by 
\begin{equation*}
  \mM\(r;\{M_l\}\):=
  \sup_{k\ge 1}\left\{\frac{r^k}{M_k}\right\}. 
\end{equation*}
\end{Def}

We shall introduce some lemmas for log convex sequences and the associated functions. For the proofs of the lemmas will be given in Appendix.

\begin{Lemm}\label{lc-prop}
Let $\{M_k\}$ be a log convex sequence. 
Then the following properties are established$:$ 
\begin{itemize}
\item[(i)]
$M_l/M_{l-1}\le lM_k/(kM_{k-1})$ for any $k>l \ge 1$.
\item[(ii)]
$M_k M_l \le M_{k+l}$. 
\item[(iii)]
There exists a constant $\de\in(0,1)$ such that 
$M_k^{1/k}\le \de M_k/M_{k-1}$. 
\item[(iv)]
$M_k/M_{k-1}-M_l/M_{l-1}\ge k-l$ for any $k>l \ge 1$. 
\item[(v)]
$\{k!M_k\}$ is log convex. 
\end{itemize}
\end{Lemm}
%
\begin{Lemm}\label{lc-binomial}
Let $\{M_k\}$ be a log convex sequence. 
For any non-negative integers $j,k,q$ and $r$ satisfying $0\le j \le k$ the following estimates are established$:$ 
\begin{equation}\label{lc-binomial-e1}
  \binom{k}{j}\frac{M_{r+j}M_{q+k-j}}{M_{q+r+k}}\le1,
\end{equation}
and
\begin{equation}\label{lc-binomial-e2}
  \binom{k}{j}\frac{M_{r+j}M_{r+k-j}}{M_{r+k}}\le
  \binom{r+k-1}{r}M_r
\end{equation}
for any $k\ge 1$. 
\end{Lemm}

\begin{Lemm}\label{lemma-af}
The associated function $\mM(r;\{M_k\})$ is strictly increasing for any $r>0$, 
\begin{equation*}
  \lim_{r\to\infty}\mM(r;\{M_k\})=\infty,
\end{equation*}
and having the following representation$:$ 
\begin{equation}\label{lemma-af-e2}
  \mM(r;\{M_k\})=
\begin{cases}
  \dfrac{r^k}{M_k}, & r\in \left[\dfrac{M_k}{M_{k-1}},\dfrac{M_{k+1}}{M_k}\right)
  \;\;
  (k=1,2,\ldots),
  \\[12pt]
  r, & r \in\left[0,\dfrac{M_1}{M_0}\right). 
\end{cases}
\end{equation}
\end{Lemm}
\begin{Lemm}\label{lemma-af1}
The following estimate is valid for any positive constant $d:$
\begin{equation*}
  \sup_{r>0}\left\{\frac{\mM\(r+d;\{M_k\}\)}{\mM\(r;\{M_k\}\)}
  \right\} \le (2e)^2\(1+d^{d+1}\).
\end{equation*}
\end{Lemm}
\begin{Lemm}\label{lemma-af2}
Let $\{M_k\}$ and $\{N_k\}$ be log convex sequences satisfying 
$M_k/M_{k-1}<N_k/N_{k-1}$ for any $k\in \N$. 
Then we have 
\begin{equation}\label{lemma-af2-e1}
  \frac{\mM\(r;\{M_k\}\)}{\mM\(r;\{N_k\}\)}>1.
\end{equation}
In particular, if $\lim_{k\to\infty}N_k/M_k=\infty$, then we have
\begin{equation}\label{lemma-af2-e2}
  \lim_{r\to\infty}\frac{\mM\(r;\{M_k\}\)}{\mM\(r;\{N_k\}\)}=\infty.
\end{equation}
Moreover, if $\lim_{k\to\infty}N_k/M_{k+1}=\infty$, then we have
\begin{equation}\label{lemma-af2-e3}
  \lim_{r\to\infty}\frac{\mM\(r;\{M_k\}\)}{r\:\mM\(r;\{N_k\}\)}=\infty.
\end{equation}
\end{Lemm}
\begin{Ex}
The associated function of the sequence in Example \ref{lc-ex} (i): 
\begin{align*}
  \cM_1(r):=\mM(r;\{\si^k k!^\nu\})
\end{align*}
satisfies 
\[
  \log \cM_1(r) \simeq r^{\frac{1}{\nu}},
\]
where $f(r)\simeq g(r)$ with positive functions $f,g$ denotes that there exist a positive constant $C>1$ such that $C^{-1} f(r) \le g(r) \le C f(r)$. 
More precisely, on $[M_k/M_{k-1},M_{k+1}/M_k)=[\si k^\nu,\si (k+1)^\nu)$ we have 
\begin{align*}
  \log\cM_1(r)
  \begin{cases}
    \ge \log\(\frac{1}{M_k}\(\frac{M_k}{M_{k-1}}\)^k\)
      =\nu\log\frac{k^k}{k!}
      =\nu k-\frac{\nu}{2}\log k +O(1),
\\
    \le \log\(\frac{1}{M_k}\(\frac{M_{k+1}}{M_k}\)^k\)
      =\nu\log\frac{(k+1)^k}{k!}
      =\nu k-\frac{\nu}{2}\log k + O(1),
  \end{cases}
\end{align*}
and 
\begin{align*}
  \nu k-\frac{\nu}{2}\log k 
  = \nu \si^{-\frac{1}{\nu}} r^{\frac{1}{\nu}} 
  - \log\sqrt{r} + O(1) 
\end{align*}
as $k\to\infty$ by Stirling's formula. 
Thus we have 
\begin{equation}\label{af-ex1}
  \cM_1(r) 
  \simeq 
  \frac{1}{\sqrt{r}}\exp\(\nu \si^{-\frac{1}{\nu}} r^{\frac{1}{\nu}}\).
\end{equation}
\end{Ex}
\begin{Rem}\label{rem-af-k!}
If $\{N_k\}$ is log convex, then 
$N_k/N_{k-1}\ge k=k!/(k-1)!$ by Lemma \ref{lc-prop} (iv). 
Therefore, by Lemma \ref{lemma-af2} and (\ref{af-ex1}) with $\nu=1$ and $\si=1$, there exists a positive constant $C$ such that 
\begin{align*}
  \log\mM\(r;\{N_k\}\) \le  \log\mM\(r;\{k!\}\) 
  \le C r
\end{align*}
for $r \ge 1$. 
\end{Rem}
\begin{Ex}\label{af-ex2}
The associated function of the sequence in Example \ref{lc-ex} (ii): 
\begin{equation*}
  \cM_2(r):=\mM\(r;\left\{\prod_{j=1}^k j\exp(j^s)\right\}\)
\end{equation*}
satisfies 
\begin{equation}\label{ex2-e}
  \log \cM_2(r) \simeq (\log(1+r))^{1+\frac{1}{s}}. 
\end{equation}
Indeed, on $[M_k/M_{k-1},M_{k+1}/M_k)=[k\exp(k^s),(k+1)\exp((k+1)^s))$, we have 
\begin{align*}
  \log\cM_2(r)
  \begin{cases}
    \ge 
       k^{s+1}\(1-\dfrac{\sum_{j=1}^k j^s}{k^{s+1}}\)
      +k\(\log k-\dfrac{\sum_{j=1}^k\log j}{k}\)
      \simeq k^{s+1},
  \\[12pt]
    \le 
       k^{s+1}\(\(\dfrac{k+1}{k}\)^s-\dfrac{\sum_{j=1}^k j^s}{k^{s+1}}\)
      +(k+1)\(\log(k+1)-\dfrac{\sum_{j=1}^{k+1}\log j}{k+1}\)
      \simeq k^{s+1}.
  \end{cases}
\end{align*}
Therefore, noting $k\simeq (\log(1+r))^{1/s}$, we have (\ref{ex2-e}). 
\end{Ex}
\begin{Rem}
By Lemma \ref{lc-prop} (i) we have $\lim_{k\to\infty}M_k/M_{k-1}=\infty$, 
it follows that the associated function $\mM(r;\{M_k\})$ increases faster than any polynomial order as $r\to\infty$ by the representation (\ref{lemma-af-e2}). 
If we adopt the same definition of the associated function for a finite sequence of positive numbers $\{M_k\}_{k=0}^l$ with $l\in \N$, then we have $\mM(r;\{M_k\})\simeq r^l$ ($r\to\infty$). 
\end{Rem}

\subsection{Extension of Manfrin's class in the ultradifferentiable class}
Let us introduce the ultradifferentiable class $\cH(\{M_k\})$ as follows: 
\begin{Def}\label{udclass}
For a log convex sequence $\{M_k\}$ and a positive real number $\rho$, 
we define the ultradifferentiable classes $\cH(\rho,\{M_k\})$ and $\cH(\{M_k\})$ by 
\begin{equation*}
  \cH(\rho,\{M_k\}):=
  \left\{ f\in L^2(\R^n)\:;\: 
  \int_{\R^n}\mM\(\frac{|\xi|}{\rho};\{M_k\}\) |\hat{f}(\xi)|^2\,d\xi<\infty
  \right\}
\end{equation*}
and
\begin{equation*}
  \cH(\{M_k\}):=\bigcup_{\rho>0}\cH(\rho,\{M_k\}).
\end{equation*}
\end{Def}
By the properties of log convex sequences and the associated functions, we immediately see the following lemma: 

\begin{Lemm}\label{lemm-uldclass}
\begin{itemize}
\item[(i)] 
If $f\in \cH(\rho,\{M_k\})$, 
then $\sup_{k\in\N}\{\rho^k M_k\||\cd|^{k/2}\hat{f}(\cd)\|^2\}<\infty$. 
\item[(ii)] 
If the log convex sequences $\{M_k\}$ and $\{N_k\}$ satisfy 
$M_k/M_{k-1} < N_{k}/N_{k-1}$ for any $k\in \N$, then 
${\cal H}(\{M_k\}) \subset {\cal H}(\{N_k\})$. 
\item[(iii)] 
$\cH(\{k!\})$, and $\cH(\{k!^\nu\})$ with $\nu>1$ are realanalytic class, and Gevrey class of order $\nu$ respectively. 
(See \cite{Kr}.)
\item[(iv)]
Let ${\cal Q}$ be the set of all log convex sequences $\{M_k\}$ satisfying the quasianalytic condition $\sum M_k/M_{k+1}=\infty$ due to the Denjoy-Carleman theorem. Then ${\cal Q}_{L^2}:=\bigcup_{\{M_k\}\in {\cal Q}}\cH(\{M_k\})$ denotes the quasianalytic class. (See \cite{Kr}.) 
\end{itemize}
\end{Lemm}

Let us consider an extension of $B^{(m)}_\Delta$ as $m\to \infty$ in the ultradifferentiable class $\cH(\{M_k\})$ as follows: 

\begin{Def}
Let $\{M_k\}$ and $\{\tM_k\}$ be log convex sequences, $\eta$ and $\rho$ be positive real numbers. 
We define $G(\rho,\eta,\{M_k\},\{\tM_k\})=G(\rho,\eta,\{M_k\},\{\tM_k\};u_0,u_1)$ as follows: 
\begin{equation*}
  G(\rho,\eta,\{M_k\},\{\tM_k\};u_0,u_1):=\int_{|\xi|\ge \rho}
  \mM\(\frac{|\xi|}{\rho};\left\{\tM_k\right\}\)
  \exp\(\frac{\eta|\xi|}{\mM\(\frac{|\xi|}{\rho};\{M_k\}\)}\)
  \cE(0,\xi)\,d\xi. 
\end{equation*}
Then we define the class of initial data $\cB_\Delta(\{M_k\},\{\tM_k\})$ by 
\begin{equation*}
  \cB_\Delta(\{M_k\},\{\tM_k\}):=\bigcup_{\eta>0}
  \left\{
  (u_0,u_1)\:;\:
  \exists\{\rho_j\}\in {\cal L},\;\;
  \sup_j\left\{G(\rho_j,\eta,\{M_k\},\{\tM_k\};u_0,u_1)\right\}<\infty
  \right\}.
\end{equation*}
\end{Def}
%
\begin{Rem}
$B^{(m)}_\Delta$ can be identified $\cB_\Delta(\{M_k\},\{M_k\})$ for a finite sequence of positive numbers $\{M_k\}_{k=0}^m$. 
Therefore, $\cB_\Delta(\{M_k\},\{\tM_k\})$ is a generalization of the Manfrin's class $B_\Delta^{(m)}$. 
\end{Rem}

The following proposition is a sort generalizations of Lemma \ref{B_Delta^m-prop} (iii) and (iv) for $\cB_\Delta(\{M_k\},\{\tM_k\})$: 
\begin{Prop}\label{Prop-cB}
Let $\{L_k\}$, $\{M_k\}$, $\{\tM_k\}$, $\{N_k\}$ and $\{\tN_k\}$ be log convex sequences. 
Then the following properties are established$:$ 
\begin{itemize}
\item[(i)]
If $(u_0,u_1)\in \cB_\Delta(\{M_k\},\{L_k\})$, then $\pa_{x_1}u_0,\ldots,\pa_{x_n}u_0,u_1\in \cH(\{L_k\})$. 
\item[(ii)]
If $M_k/M_{k-1}<N_k/N_{k-1}$ then $\cB_\Delta(\{N_k\},\{\tM_k\})\subset \cB_\Delta(\{M_k\},\{\tM_k\})$, and if $\tM_k/\tM_{k-1}<\tN_k/\tN_{k-1}$ then $\cB_\Delta(\{M_k\},\{\tM_k\})\subset \cB_\Delta(\{M_k\},\{\tN_k\})$. 
\item[(iii)]
If $L_k/L_{k-1}<\tM_k/\tM_{k-1}$ and $\lim_{k\to\infty}\tM_k/L_k=\infty$, then there exists $\{u_{1,j}\}_{j=1}^\infty$ such that $\{(0,u_{1,j})\}_{j=1}^\infty$ is uniformly bounded in $\cB_\Delta(\{M_k\},\{\tM_k\})$, but $\{u_{1,j}\}_{j=1}^\infty$ is not bounded in $\cH(\{L_k\})$ as $j\to\infty$. 
\item[(iv)] 
If $N_k/M_{k+1}\nearrow \infty$ as $k\to\infty$, then there exists $\{u_{1,j}\}_{j=1}^\infty$ such that 
$\{(0,u_{1,j})\}_{j=1}^\infty$ is uniformly bounded in $\cB_\Delta(\{M_k\},\{\tM_k\})$, but not bounded in $\cB_\Delta(\{N_k\},\{\tN_k\})$ as $j\to\infty$. 
\end{itemize}
\end{Prop}

\begin{Rem}
We observe from Proposition \ref{Prop-cB} (iv) that $\cB_\Delta(\{M_k\},\{\tM_k\})$ is not included in $B_\Delta^{(m)}$ for any $m\in \N$. 
\end{Rem}
\begin{Thm}\label{MThm}
Let $\{L_k\}$ be a log convex sequence. 
If $(u_0,u_1)\in \cB_\Delta(\{k! L_k\},\{L_k\})$, 
then (\ref{K}) has a global strong solution. 
Moreover, for any $t>0$ there exists $\rho=\rho(t)>0$ such that 
\begin{equation}\label{MThm-e1}
  \int_{\R^n}\mM\(\frac{|\xi|}{\rho};\{L_k\}\) \cE(t,\xi)\,d\xi<\infty. 
\end{equation}
\end{Thm}

\begin{Rem}\label{rem-strongsol}
(\ref{MThm-e1}) provides $\||\cd|^{k}\hat{u}_t(t,\cd)\|<\infty$ for any $t\in[0,\infty)$ and any $k=0,1,\ldots$, which implies that $u(t,x)$ is a strong solution of (\ref{K}). 
\end{Rem}
%
\begin{Rem}
Noting Proposition \ref{Prop-cB} (i), $\{L_k\}$ should satisfy the non-quasianalytic condition: $\sum L_k/L_{k+1}<\infty$; 
otherwise $\cB_\Delta(\{k! L_k\},\{L_k\})$ is included in Nishihara's quasianalytic class. 
\end{Rem}

\begin{Rem}
If $\{M_k\}$ is log convex, then the associated function $\mM(r;\{M_k\})$ is convex by the representation (\ref{lemma-af2-e2}).  
Therefore, the weight functions which define the ultradifferentiable class $\cH(\{M_k\})$ are convex; thus the quasianalytic class ${\cal Q}_{L^2}$ due to \cite{Ni} is defined by convex weight functions. 
Recently, Ghisi-Gobbino \cite{GG} proved the global solvability in a wider class of Nishihara's quasianalytic class. 
Briefly, their class is defined by monotone increasing weight functions satisfying a sort of quasianalytic condition, but not necessary to be convex. 
On the other hand, $\cB_\Delta(\{M_k\},\{\tM_k\})$ cannot be defined by monotone increasing weight functions. 
At present, \cite{GG} and Theorem \ref{MThm} provide the widest classes for the global solvability of (\ref{K}) with large data. 
\end{Rem}

\section{Estimates of a linearized problem}
\subsection{Linear wave equation with time dependent coefficient}
Let us consider the following Cauchy problem of the wave equation with time dependent coefficient: 
\begin{equation}\label{Lu}
\begin{cases}
  \pa_t^2 u(t,x)-a(t)^2 \Delta u(t,x)=0,
  \quad
  (t,x)\in(0,T)\times \R^n,
\\
  u(0,x)=u_0(x),
  \;\;
  (\pa_t u)(0,x)=u_1(x),
  \quad
  x\in \R^n, 
\end{cases}
\end{equation}
where $a(t)\in C^\infty([0,T))$ satisfies 
\begin{equation}\label{a0}
  1\le a(t) \le a_1 
\end{equation}
for a positive constant $a_1$. 
If the local solution of (\ref{K}) satisfies (\ref{MThm-e1}) on $[0,T)$, 
then $\Phi(t):=1+\|\nabla u(t,\cd)\|^2 \in C^\infty([0,T))$; thus (\ref{Lu}) is a linearized problem of (\ref{K}) adopting $a(t)=\sqrt{\Phi(t)}$. 

By partial Fourier transformation with respect to $x\in \R^n$, 
(\ref{Lu}) is reduced to the following problem: 
\begin{equation}\label{L}
\begin{cases}
  \pa_t^2 v(t,\xi)+a(t)^2|\xi|^2 v(t,\xi)=0,
  \quad
  (t,\xi)\in(0,T)\times \R^n,
\\
  v(0,\xi)=\hat{u}_0(\xi),
  \;\; 
  (\pa_t v)(0,\xi)=\hat{u}_1(\xi),
  \quad
  \xi\in \R^n,
\end{cases}
\end{equation}
where $v(t,\xi)=\hat{u}(t,\xi)$. 
We define the energy functional $\cE(t,\xi)=\cE(t,\xi;v)$ by 
\begin{equation}\label{cE}
  \cE(t,\xi;v):=\frac12\(|\xi|^2|v(t,\xi)|^2+|\pa_t v(t,\xi)|^2\). 
\end{equation}
The main purpose of this section is to prove the following proposition: 
\begin{Prop}\label{PropL}
Let $a(t)\in C^\infty([0,T))$ satisfy $(\ref{a0})$. 
If there exist a log convex sequence $\{M_k\}$ 
and a positive constant $\mu_0$ such that 
\begin{equation}\label{PropL-e1}
  \left|a^{(k)}(t)\right|\le \mu_0^k M_k
\end{equation}
for any $k\in \N$, 
then there exist positive constants $\ka_0$ and $C_0$ such that 
for any $\ka\ge \ka_0$ and any $m\in \N$ 
the following estimate is established$:$
\begin{equation}\label{PropL-e2}
  \cE(t,\xi) \le C_0\cE(T_0,\xi)\exp\((T-T_0)(\ka\mu_0)^m M_m |\xi|^{-m+1}\)
\end{equation}
for $0\le T_0 < t < T$ and $|\xi|\ge \ka\mu_0 M_m/M_{m-1}$. 
\end{Prop}
\begin{Rem}
The corresponding estimates of (\ref{PropL-e2}) to the linear problem (\ref{L}) with $a(t)\in C^m([0,\infty))$ were proved in \cite{H06, M05}; however, these estimates has no meaning for the asymptotic as $m\to\infty$. 
\end{Rem}

The proof of Proposition \ref{PropL} consists of  three parts; 
reduction to first order system and diagonalization, 
symbol calculus, and these applications for the estimate (\ref{PropL-e2}). 
The original idea of the proof was introduced in \cite{HI13}, and the following proof is a modification of it. 
\subsection{Refined diagonalization} 
Let $v=v(t,\xi)$ be the solution of (\ref{L}), and define 
$V_1={}^t(v_t+\i a(t)|\xi|v , v_t -\i a(t)|\xi|v)$. 
Then the equation of (\ref{L}) is reduced to the following first order system:
\begin{equation}\label{V1}
\pa_t V_1=A_1 V_1,
\quad
A_1=\begin{pmatrix} \phi_1 & \ol{b_1} \\ b_1 & \ol{\phi_1} \end{pmatrix},
\end{equation}
where
\begin{equation*}
  b_1=-\frac{a'(t)}{2a(t)}
  \;\;\text{ and }\;\;
  \phi_1=\frac{a'(t)}{2a(t)}+\i a(t)|\xi|.
\end{equation*}
Let $\la_1$ and $\ol{\la_1}$ be the eigenvalues of $A_1$ represented by 
\begin{equation*}
  \la_1=\phi_{1\Re}+\i\sqrt{\phi_{1\Im}^2-|b_1|^2}, 
\end{equation*}
where we denote $\Re\phi=\phi_\Re$ and $\Im\phi=\phi_\Im$. 
Then the corresponding eigenvectors are given by ${}^t(1,\th_{1})$ and ${}^t(\ol{\th_{1}},1)$, where 
\begin{equation*}
  \th_{1}
 =\frac{\la_1-\phi_1}{\ol{b_1}}
 =-\i\frac{\phi_{1\Im}}{\ol{b_1}}
  \(1-\sqrt{1-\frac{|b_1|^2}{\phi_{1\Im}^2}}\). 
\end{equation*}
Thus $A_1$ is diagonalized by the diagonalizer $\Th_1$ as follows: 
\begin{equation*}
  \Th_1^{-1}A_1 \Th_1 =
  \begin{pmatrix} \la_{1} & 0 \\ 0 & \ol{\la_{1}} \end{pmatrix},
  \;\text{ where }\;
  \Th_1=\begin{pmatrix} 1 & \ol{\th_{1}}\\ \th_{1} & 1\end{pmatrix}
\end{equation*}
since $|\th_1|<1$. 
Therefore, denoting $V_2=\Th_1^{-1}V_1$, (\ref{V1}) is reduced to the following system: 
\begin{equation}\label{V2}
  \pa_t V_2 = A_2 V_2,
\end{equation}
where
\begin{align*}
  A_2=\begin{pmatrix} \la_1 & 0 \\ 0 & \ol{\la_1} \end{pmatrix}
      -\Th_1^{-1}(\pa_t \Th_1)
     =\begin{pmatrix} \phi_2 & \ol{b_2} \\ b_2 & \ol{\phi_2} \end{pmatrix},
\end{align*}
\begin{align*}
  b_2=-\frac{(\th_1)_t}{1-|\th_1|^2}
  \;\; \text{ and } \;\;
  \phi_2=\la_1+\frac{\ol{\th_1}(\th_1)_t}{1-|\th_1|^2}. 
\end{align*}
Generally, the diagonalization procedure above is represented by the following lemma: 
\begin{Lemm}\label{r-diag}
Let $V_k$ be a solution to the following system$:$ 
\begin{equation*}
  \pa_t V_k = A_k V_k,
  \quad
  A_k=\begin{pmatrix} \phi_k & \ol{b_k} \\ b_k & \ol{\phi_k} \end{pmatrix}, 
\end{equation*}
and $\Th_k$ be the diagonalizer of $A_k$ defined by 
\begin{equation*}
  \Th_k:=\begin{pmatrix}
    1 & \ol{\th_k}\\ \th_k & 1
  \end{pmatrix}
  \; \text{ and } \;
  \th_k:=-\i\frac{\phi_{k\Im}}{\ol{b_k}}
  \(1-\sqrt{1-\frac{|b_k|^2}{\phi_{k\Im}^2}}\).
\end{equation*}
If $|\th_k|<1$, then $V_{k+1}=\Th_k^{-1} V_k$ solves the following system$:$ 
\begin{align*}
  \pa_t V_{k+1} = A_{k+1} V_{k+1},
  \quad
  A_{k+1}=\begin{pmatrix} \phi_{k+1} & \ol{b_{k+1}} \\ 
  b_{k+1} & \ol{\phi_{k+1}} \end{pmatrix}, 
\end{align*}
where
\begin{equation}\label{b_k+1}
  b_{k+1}=-\frac{(\th_k)_t}{1-|\th_k|^2},
\end{equation}
\begin{equation}\label{phi_k+1-R}
  \phi_{(k+1)\Re}=\phi_{k\Re} - \pa_t \log\sqrt{1-|\th_k|^2}
\end{equation}
and
\begin{equation}\label{phi_k+1-I}
  \phi_{(k+1)\Im}=
  \sqrt{\phi_{k\Im}^2-|b_k|^2}-\Im\{\ol{\th_k}b_{k+1}\}.
\end{equation}
\end{Lemm}
\begin{proof}
The proof is straightforward if we note that the eigenvalues $\{\la_k,\ol{\la_k}\}$, and their corresponding eigenvectors $\{{}^t(1,\th_{k}),{}^t(\ol{\th_{k}},1)\}$ of $A_k$ are represented by 
\begin{equation*}
  \la_k=\phi_{k\Re}+\i\sqrt{\phi_{k\Im}^2-|b_k|^2}
\end{equation*}
and
\begin{equation*}
  \th_{k}
 =\frac{\la_k-\phi_k}{\ol{b_k}}
 =-\i\frac{\phi_{k\Im}}{\ol{b_k}}
  \(1-\sqrt{1-\frac{|b_k|^2}{\phi_{k\Im}^2}}\). 
\end{equation*}
\end{proof}

By applying Lemma \ref{r-diag} successively, (\ref{V1}) is reduced to the following equation: 
\begin{equation}\label{Vm}
  \pa_t V_m = A_k V_m,
  \quad
  A_m=\begin{pmatrix} \phi_m & \ol{b_m} \\ b_m & \ol{\phi_m} \end{pmatrix}. 
\end{equation}
However, the diagonalization procedure by Lemma \ref{r-diag} is only formal because the invertibility of the diagonalizer $\Th_k$ are not ensured. 
We shall consider this problem in the next section to introduce some symbol classes. 

\subsection{Symbol classes}
Let $\{M_k\}$ be a log convex sequence and $\rho$ a positive constant. 
For $m\in \N$ we define $Z_H(m,\rho)$ by 
\begin{equation*}
  Z_H(m,\rho):=\left\{(t,\xi)\in [0,\infty)\times \R^n\:;\:
  |\xi|\ge \rho \frac{M_m}{M_{m-1}} \right\}. 
\end{equation*}
Let $\mu$ be a positive constant to be defined by (\ref{mu}). 
For integers $p$, $q$, $r$ satisfying $0\le p \le m$, 
and a positive real number $K$, we define the symbol classes $S^{(p)}\{q,r,K\}$ as the set of all functions satisfying 
\begin{equation*}
  \left|\pa_t^k f(t,\xi)\right|\le
  K \frac{\mu^{r+k}M_{r+k} }{(r+k+1)^2}
  |\xi|^q
  \;\;(k=0,\ldots,p)
\end{equation*}
in $Z_H(m,\rho)$. 
In particular, we denote $S^{(p)}\{q,r,1\}=S^{(p)}\{q,r\}$ without any confusion. 

We immediately see that $S^{(p_1)}\{q,r,K\}\subset S^{(p_2)}\{q,r,K\}$ for $p_1>p_2$ from the definition. 
Moreover, we have the following properties: 

\begin{Lemm}\label{symbol-p1}
The following properties are established in $Z_H(m,\rho)$ for $\ka_1=3\pi^2:$ 
\begin{itemize}
\item[(i)]
 If $f\in S^{(p)}\{q,r,K\}$ and $p\ge 1$, then $\pa_t f \in S^{(p-1)}\{q,r+1,K\}$. 

\item[(ii)] 
If $f_1 \in S^{(p)}\{q,r,K_1\}$ and $f_2 \in S^{(p)}\{q,r,K_2\}$, then 
$f_1 + f_2 \in S^{(p)}\{q,r,K_1+K_2\}$. 

\item[(iii)]
 If $f \in S^{(p)}\{q,r,K_1\}$ and $K_2>0$, then $K_2 f\in S^{(p)}\{q,r,K_1K_2\}$. 

\item[(iv)]
If $f_1\in S^{(p)}\{q_1,r_1,K_1\}$ and $f_2\in S^{(p)}\{q_2,r_2,K_2\}$, 
then $f_1f_2\in S^{(p)}\{q_1+q_2,r_1+r_2,\ka_1 K_1K_2\}$ 
for $r_1,r_2\ge 0$. 
In particular, if $f_2$ is independent of $t$, then 
$f_1f_2\in S^{(p)}\{q_1+q_2,r_1+r_2,K_1K_2\}$. 

\item[(v)]
If $f\in S^{(p)}\{q,r,K\}$, then 
$f\in S^{(\min\{p,m-r\})}\{q+l,r-l,K(\mu/\rho)^l\}$ for $1\le l\le r\le m$. 

\item[(vi)]
If $f_1\in S^{(p)}\{q,r,K_1\}$ and $f_2\in S^{(p)}\{-r,r,K_2\}$, 
then $f_1f_2\in S^{(p)}\{q,r,\ka_1 K_1K_2(2\mu/\rho)^{r}\}$ 
for $p\le m$ and $1\le r\le m$. 
\end{itemize}
\end{Lemm}
\begin{proof}
(i), (ii) and (iii) are evident from the definition of the symbol classes. 

\noindent
(iv): 
Let $k\in \N$ and assume that $r_1\le r_2$ without loss of generality. 
By the inequality
\begin{equation}\label{sum_k^-2}
  \sum_{j=0}^k \(\frac{r_1+r_2+k+j+1}{(r_1+j+1)(r_2+k-j+1)}\)^2
\le
  \ka_1,
\end{equation}
which will be proved in Appendix, 
Leibniz rule, Lemma \ref{lc-binomial} and (\ref{sum_k^-2}), we have
\begin{align*}
\left|\pa_t^k\(f_1 f_2\)\right|
\le& K_1K_2
  \frac{\mu^{r_1+r_2+k} M_{r_1+r_2+k}}{(r_1+r_2+k+1)^2}
  |\xi|^{q_1+q_2}
\\
  &\times
  \sum_{j=0}^k \binom{k}{j}
  \frac{M_{r_2+k-j} M_{r_1+j}}{M_{r_1+r_2+k}}
  \(\frac{r_1+r_2+k+j}{(r_1+j+1)(r_2+k-j+1)}\)^2.
\\
\le& \ka_1K_1K_2
  \frac{\mu^{r_1+r_2+k} M_{r_1+r_2+k}}{(r_1+r_2+k+1)^2}
  |\xi|^{q_1+q_2}.
\end{align*}

\noindent
(v): Let $0\le k \le \min\{p,m-r\}$. 
Then we have 
\begin{align*}
  \left|\pa_t^k f\right|
  \le& K\frac{\mu^{r+k} M_{r+k}}{(r+k+1)^2} |\xi|^{q}
\\
  =& 
  K\frac{\mu^{r-l+k} M_{r-l+k}}{(r-l+k+1)^2} |\xi|^{q+l}
    \(\frac{\mu}{|\xi|}\)^l
    \frac{M_{r+k}}{M_{r+k-1}}\cdots 
      \frac{M_{r-l+k+1}}{M_{r-l+k}}
    \(\frac{r-l+k+1}{r+k+1}\)^2 \\
\le& 
  K\frac{\mu^{r-l+k} M_{r-l+k}}{(r-l+k+1)^2} |\xi|^{q+l}
    \(\frac{\mu M_{m-1}}{\rho M_m}\)^l
    \(\frac{M_{r+k}}{M_{r+k-1}}\)^{l}
\\
\le& 
  K\(\frac{\mu}{\rho}\)^l \frac{\mu^{r-l+k} M_{r-l+k}}{(r-l+k+1)^2} 
  |\xi|^{q+l}.
\end{align*}

\noindent
(vi): Let $0\le k \le p$. 
By Lemma \ref{lc-binomial} and (\ref{sum_k^-2}) we have
\begin{align*}
  \left|\pa_t^k (f_1f_2)\right|
\le& 
  K_1K_2
  \frac{\mu^{r+k} M_{r+k}}{(r+k+1)^2}
  |\xi|^{q}
  \(\frac{\mu}{|\xi|}\)^{r}
  \sum_{j=0}^k \binom{k}{j}
  \frac{M_{r+j}M_{r+k-j}}{M_{r+k}} 
  \frac{(r+k+1)^2}{(r+j+1)^2 (r+k-j+1)^2} 
\\
\le& 
  K_1K_2
  \frac{\mu^{r+k} M_{r+k}}{(r+k+1)^2}
  |\xi|^{q}
\\
  &\times
  \(\frac{\mu}{\rho}\frac{M_{m-1}}{M_m}\)^{r}
  \binom{r+k-1}{r} 
  M_{r}\sum_{j=0}^k 
  \frac{(r+k+1)^2}{(r+j+1)^2 (r+k-j+1)^2}
\\
\le& 
  \ka_1 K_1K_2\(\frac{\mu}{\rho}\)^{r}
  \frac{\mu^{r+k} M_{r+k}}{(r+k+1)^2}
  |\xi|^{q}
  \(\frac{M_{m-1}}{M_m}\)^{r}
  \binom{r+k-1}{r} 
  M_{r}
\\
\le& 
  \ka_1 K_1K_2\(\frac{2\mu}{\rho}\)^{r}
  \frac{\mu^{r+k} M_{r+k}}{(r+k+1)^2}
  |\xi|^{q},
\end{align*}
where we used the following estimates: 
\begin{align*}
  \(\frac{M_{m-1}}{M_m}\)^{r} \binom{r+k-1}{r} M_{r}
=&
  \frac{(r+k-1)!}{m^r(k-1)!}
  \frac{\frac{M_{r}}{rM_{r-1}}\cdots\frac{M_1}{1M_0}}
       {\(\frac{M_m}{mM_{m-1}}\)^{r}} 
\le
  \frac{(r+k-1)\cdots k}{m^r}
\le
  2^r.
\end{align*}
\end{proof}
\begin{Lemm}\label{symbol-p2}
Let $m$, $p$, $r$ be positive integers satisfying $\max\{p,r\}\le m$ 
and $K$ a positive real number. 
If $f\in S^{(p)}\{-r,r,K\}$, $\rho>\mu$ and $\rho\ge 4\ka_1 \mu K$, 
then there exist $g_1, g_2 \in S^{(\min\{p,m-r\})}\{-r,r,2K\}$ such that the following properties are established in $Z_H(m,\rho):$
\begin{equation*}
  \frac{1}{1-f}=1+g_1
\end{equation*}
and
\begin{equation}\label{symbol-p2-e2}
  1-\sqrt{1-f} = \frac12 f(1+g_2).
\end{equation}
\end{Lemm}
\begin{proof}
We denote $p_0=\min\{p,m-r\}$. 
By using Lemma \ref{symbol-p1} (v) with $q=-r$ and $l=r$, we have 
$f\in S^{(p_0)}\{0,0,K(\mu/\rho)^r\}$, 
it follows that $|f| \le K(\mu/\rho)^r<1$ for $K \mu<\rho$. 
Moreover, by applying Lemma \ref{symbol-p1} (vi) with $f_1=f$, $f_2=f^j$, 
$q=-r$, $K_1=K$ and $K_2=K(\ka_1K(2\mu/\rho)^k)^{j-1}$ for $j=1,\ldots,l-1$, we have 
\begin{equation*}
  f^l\in S^{(p_0)}\left\{-r,r, K\(K\ka_1\(\frac{2\mu}{\rho}\)^r\)^{l-1}\right\}
\end{equation*}
for $l=2,3,\ldots$. 
Therefore, by Lemma \ref{symbol-p1} (ii) and noting 
\begin{align*}
  \sum_{l=1}^\infty \(K\ka_1\(\frac{2\mu}{\rho}\)^r\)^{l-1}
  \le
  \sum_{l=1}^\infty \(\frac{2K\ka_1\mu}{\rho}\)^{l-1}
 =\frac{1}{1-\frac{2K\ka_1\mu}{\rho}}
 \le 2, 
\end{align*}
we have 
\begin{align*}
  g_1=\sum_{l=1}^\infty f^l
  \in S^{(p_0)}\left\{-r,r, 2K\right\}. 
\end{align*}
Moreover, thanks to the representation 
\begin{align*}
  g_2=2\sum_{l=1}^\infty \binom{1/2}{l+1}(-f)^{l}
\end{align*}
and the inequality 
$|\binom{1/2}{l+1}|\le 1/2$ for any $l\ge 0$, we have 
(\ref{symbol-p2-e2}). 
\end{proof}

For $f\in S^{(p)}\{q,r,K\}$ we introduce the following notation 
for convenience: 
\begin{equation*}
  f \lesssim K\si^{(p)}\{q,r\}. 
\end{equation*}
In particular, we denote $1\si^{(p)}\{q,r\}=\si^{(p)}\{q,r\}$, 
that is, $\si^{(p)}\{q,r\}$ stands for any function in 
the symbol class $S^{(p)}\{q,r,1\}$. 
Moreover, for positive real numbers $K_1$ and $K_2$ we introduce the following notations: 
\begin{itemize}
\item 
$K_1 (K_2\si^{(p)}\{q,r\})=(K_1 K_2)\si^{(p)}\{q,r\}$. 
\item 
$K_1\si^{(p)}\{q,r\}+K_2\si^{(p)}\{q,r\}=(K_1+K_2)\si^{(p)}\{q,r\}$. 
\item 
$K_1\si^{(p_1)}\{q_1,r_1\}\lesssim K_2\si^{(p_2)}\{q_2,r_2\}
  \; \Leftrightarrow \;
  S^{(p_1)}\{q_1,r_1,K_1\}\subseteq S^{(p_2)}\{q_2,r_2,K_2\}$. 
\item 
$(K_1\si^{(p)}\{q_1,r_1\})(K_2\si^{(p)}\{q_2,r_2\})
  \lesssim \ka_1 K_1K_2\si^{(p)}\{q_1+q_2,r_1+r_2\}$. 
\end{itemize}
By use of the above notation the properties of Lemma \ref{symbol-p1} and Lemma \ref{symbol-p2} 
are expressed as follows: 

\begin{Lemm}\label{symbol-p3}
Let $0\le r \le m$. 
Then the following properties are established in $Z_H(m,\rho):$ 
\begin{itemize}
\item[(i)] 
If $p_1>p_2$, then $\si^{(p_1)}\{q,r\}\lesssim \si^{(p_2)}\{q,r\}$. 
\item[(ii)] 
If $p\ge 1$ and $f\lesssim \si^{(p)}\{q,r\}$, then $\pa_t f\lesssim \si^{(p-1)}\{q,r+1\}$.
\item[(iii)] 
$\si^{(p)}\{q,r\}\lesssim (\mu/\rho)^l\si^{(\min\{p,m-r\})}\{q+l,r-l\}$  for $1\le l\le r \le m$. 
\item[(iv)] 
$\si^{(p)}\{q_1,r_1\} \si^{(p)}\{q_2,r_2\}
\lesssim \ka_1 \si^{(p)}\{q_1+q_2,r_1+r_2\}$ for $r_1,r_2\ge 0$, 
and $\si^{(p)}\{q,r\} \si^{(p)}\{-r,r\}
\lesssim \ka_1(2\mu/\rho)^r \si^{(p)}\{q,r\}$ for any $p\le m$ and 
$0\le r\le m$. 
\item[(v)]
If $f\lesssim K\si^{(p)}\{-r,r\}$, 
$\rho>\mu$, $\rho\ge 4\ka_1\mu K$ and $\max\{p,r\}\le m$, 
then $1/(1-f) \lesssim 1+ 2K\si^{(\min\{p,m-r\})}\{-r,r\}$. 
Moreover, if $f\not=0$, then 
$2(1-\sqrt{1-f})/f \lesssim 1+2K\si^{(\min\{p,m-r\})}\{-r,r\}$. 
\end{itemize}
\end{Lemm}

\subsection{Estimates in the symbol classes}
%
Let us determine the constant $\mu$ by 
\begin{equation}\label{mu}
  \mu:=e^2\mu_0 \ka_1. 
\end{equation}
By using the properties of the symbol classes above, we have the following lemma: 
\begin{Lemm}\label{symbol_b1-th1}
For $\rho \ge \sqrt{2}\ka_1^2 \mu$ the following estimates are established$:$ 
\begin{equation}\label{symbol_b1-th1-e1}
  b_1 \lesssim \frac{1}{2}\si^{(m-1)}\{0,1\}
\end{equation}
and
\begin{equation}\label{symbol_b1-th1-e2}
  \th_1\lesssim \frac{\ka_1}{2} \si^{(m-1)}\{-1,1\}. 
\end{equation}
\end{Lemm}
\begin{proof}
By (\ref{PropL-e1}) we see that 
\begin{equation}\label{ak-m}
  \left|a^{(k)}(t)\right|
  = \(\(k+1\)^{\frac{2}{k}} \mu_0\)^k \frac{M_k}{(k+1)^2} 
  \le \frac{\(e^2 \mu_0\)^k M_k}{(k+1)^2}
\end{equation}
for any $k\in \N$. 
Moreover, we have 
\begin{equation}\label{est_a^(-1)}
  \left|\(\frac{1}{a}\)^{(k)}\right|
  \le \frac{\(e^2\ka_1 \mu_0\)^k M_k}{(k+1)^2} 
\end{equation}
for any $k\in \N\cup\{0\}$, 
it follows that 
\begin{equation}\label{1/a-symb}
  \frac{1}{a} \lesssim \si^{(m)}\{0,0\}. 
\end{equation}
$1/a\lesssim \si^{(p)}\{0,0\}$ for any $p\in \N$. 
Indeed, (\ref{est_a^(-1)}) is trivial for $k=0$ by (\ref{a0}). 
Suppose that (\ref{est_a^(-1)}) is valid for any $0\le k\le l$. 
Then by virtue of Leibniz rule we see that 
\begin{align*}
  0=&\left|\frac{d^{l+1}}{dt^{l+1}} 1\right|
   =\sum_{j=0}^{l+1}\binom{l+1}{j}a^{(j)}\(\frac{1}{a}\)^{(l-j+1)},
\end{align*}
and thus
\begin{align*}
  \(\frac{1}{a}\)^{(l+1)}
 =-\frac{1}{a}\sum_{j=1}^{l}
     \binom{l+1}{j}a^{(j)}
     \(\frac{1}{a}\)^{(l-j+1)}
  -\frac{a^{(l+1)}}{a^2}.
\end{align*}
By Lemma \ref{lc-binomial}, (\ref{sum_k^-2}) and (\ref{ak-m}) we have 
\begin{align*}
  \left|\(\frac{1}{a}\)^{(l+1)}\right|
\le&
  \sum_{j=1}^{l}\binom{l+1}{j}
    \frac{\(e^2 \mu_0\)^{j} M_j}{(j+1)^2}
    \frac{\(e^2 \ka_1\mu_0\)^{l-j+1} M_{l-j+1}}{(l-j+2)^2}
 +\frac{\(e^2\mu_0\)^{l+1}M_{l+1}}{(l+2)^2}
\\
\le&
  \frac{\(e^2\ka_1 \mu_0\)^{l+1} M_{l+1}}{(l+2)^2}
  \(\sum_{j=1}^{l}\ka_1^{-j}\binom{l+1}{j}
    \frac{M_j M_{l-j+1}}{M_{l+1}}
    \frac{(l+2)^2}{(j+1)^2(l-j+2)^2}
  +\ka_1^{-l-1}\)
\\
\le&
  \frac{\(e^2\ka_1 \mu_0\)^{l+1} M_{l+1}}{\ka_1(l+2)^2}
  \sum_{j=1}^{l+1}
    \frac{(l+2)^2}{(j+1)^2(l-j+2)^2}
\\
\le&
  \frac{\(e^2\ka_1 \mu_0\)^{l+1} M_{l+1}}{(l+2)^2}.
\end{align*}
Thus the estimate (\ref{est_a^(-1)}) is valid for $k=l+1$, 
it follows that (\ref{est_a^(-1)}) is valid for any $k\in \N$. 
Therefore, by Lemma \ref{lc-binomial}, \ref{symbol-p3} (iv), 
(\ref{sum_k^-2}) and (\ref{ak-m}) we have 
\begin{align*}
  \left|b_1^{(k)}(t)\right|
  \le& \frac12\sum_{j=0}^{k} \binom{k}{j}
    \left|a^{(j+1)}\right|\left|\(\frac{1}{a}\)^{(k-j)}\right|
\\
  \le& 
  \frac{1}{2}
  \sum_{j=0}^{k} \binom{k}{j}
  \frac{\(e^2\mu_0\)^{j+1}M_{j+1}}{(j+2)^2}
  \frac{\(e^2\ka_1\mu_0\)^{k-j}M_{k-j}}{(k-j+1)^2}
\\
  =& 
  \frac{1}{2}\frac{\mu^{k+1} M_{k+1}}{(k+2)^2}
  \sum_{j=0}^{k} \binom{k}{j}
  \frac{M_{j+1}M_{k-j}}{M_{k+1}}
  \frac{(k+2)^2}{(j+2)^2(k-j+1)^2}
  \frac{\ka_1^{k-j}\(e^2\mu_0\)^{k+1}}{\mu^{k+1} }
\\
  \le& 
  \frac{1}{2}\frac{\mu^{k+1} M_{k+1}}{(k+2)^2}
  \(\frac{e^2\ka_1\mu_0}{\mu}\)^{k+1}
  \le 
  \frac12\frac{\mu^{k+1} M_{k+1}}{(k+2)^2}
\end{align*}
for any $k\in \N_0:=\N\cup\{0\}$; 
hence (\ref{symbol_b1-th1-e1}) is valid. 
By (\ref{1/a-symb}) and Lemma \ref{symbol-p1} (iv) we have 
\begin{equation}\label{symbol_1/phi1}
  \frac{1}{\phi_{1\Im}} \lesssim \si^{(m)}\{-1,0\}, 
\end{equation}
thus by Lemma \ref{symbol-p3} (iv) we have 
\begin{equation*}
  \frac{|b_1|^2}{\phi_{1\Im}^2}
  \lesssim  
  \frac{\ka_1^2}{4}\si^{(m-1)}\{-1,1\}\si^{(m-1)}\{-1,1\}
  \lesssim  
  \frac{\ka_1^3\mu}{2\rho}\si^{(m-1)}\{-1,1\}. 
\end{equation*}
Consequently, by Lemma \ref{symbol-p3} (iv) and (v) 
for $\rho\ge \sqrt{2}\ka_1^2\mu(>\mu)$, 
which corresponds to $\rho\ge 4\ka_1\mu (\ka_1^3\mu/(2\rho))$, 
we have 
\begin{align*}
  \th_1
 =&\frac{-\frac{\i b_1}{\phi_{1\Im}}
    \(1-\sqrt{1-\frac{|b_1|^2}{\phi_{1\Im}^2}}\)}
    {\frac{|b_1|^2}{\phi_{1\Im}^2}}
\\
 \lesssim&
  \frac{\ka_1}{2}\si^{(m-1)}\{-1,1\} 
  \(\frac12+ \frac{\ka_1^3 \mu}{2\rho} \si^{(m-1)}\{-1,1\}\)
\\
 \lesssim&
  \frac{\ka_1}{4}\si^{(m-1)}\{-1,1\} 
 +\frac{\ka_1^4 \mu}{4\rho}\si^{(m-1)}\{-1,1\}\si^{(m-1)}\{-1,1\}
\\
 \lesssim&
  \(\frac{\ka_1}{4} + \frac{\ka_1^5 \mu^2}{2\rho^2}\)
  \si^{(m-1)}\{-1,1\}
 \lesssim
  \frac{\ka_1}{2}\si^{(m-1)}\{-1,1\}. 
\end{align*}
\end{proof}
By (\ref{symbol_b1-th1-e2}), Lemma \ref{symbol-p3} (iii) and (iv), we immediately see that 
\begin{equation*}
  |\th_1|^2 \lesssim \frac{\ka_1^3\mu^2}{2\rho^2}\si^{(m-1)}\{0,0\}
  \lesssim \frac{1}{4\ka_1}\si^{(m-1)}\{0,0\}, 
\end{equation*}
it follows that $|\th_1|<1$ in $Z_{H}(m,\rho)$. 
Therefore, (\ref{V1}) is actually reduced to (\ref{V2}). 
We shall show that $\Th_k$ ($k=2,\ldots,m-1$) are also invertible uniformly with respect to $k$, and thus we will come up to the equation (\ref{Vm}).

\begin{Lemm}\label{est-bk-th_k}
For $\nu:=4\ka_1^4$ and $\rho\ge \tilde{\rho}:=32\ka_1^3\nu^2\mu$ the following estimates are established$:$ 
\begin{equation}\label{est-bk-th_k-e1}
  b_k \lesssim \nu^{k}\si^{(m-k)}\{-k+1,k\}
\end{equation}
and
\begin{equation}\label{est-bk-th_k-e2}
  \th_k \lesssim \nu^{k}\si^{(m-k)}\{-k,k\}
\end{equation}
for any $k=1,\ldots,m$ in $Z_{H}(m,\rho)$. 
\end{Lemm}
\begin{proof}
Suppose that $\th_k \lesssim \nu^k \si^{(m-k)}\{-k,k\}$. 
By Lemma \ref{symbol-p3} (iii) and (iv), we have 
\begin{align*}
  |\th_k|^2 
  \lesssim&
   \ka_1\nu^{2k} \(\frac{2\mu}{\rho}\)^k \si^{(m-k)}\{-k,k\}
   \lesssim
   \ka_1 \(\frac{2\nu^2\mu^2}{\rho^2}\)^{k} \si^{(m-k)}\{0,0\}
\\
   \lesssim&
   \frac{2\ka_1\nu^2\mu^2}{\rho^2} \si^{(m-k)}\{0,0\} 
   \lesssim
   \frac12\si^{(m-k)}\{0,0\},  
\end{align*}
and thus 
\begin{align*}
  \frac{1}{1\pm|\th_k|^2}
  \lesssim 1+\si^{(m-k)}\{0,0\}
  \lesssim 2\si^{(m-k)}\{0,0\}
\end{align*}
by Lemma \ref{symbol-p3} (v). 
Therefore, by (\ref{b_k+1}), Lemma \ref{symbol-p3} (ii) and (iv) we have 
\begin{equation}\label{b_k+1-1}
  b_{k+1}\lesssim 
  \(\nu^k \si^{(m-k-1)}\{-k,k+1\}\)\(2\si^{(m-k)}\{0,0\}\)
  \lesssim 2\ka_1 \nu^k \si^{(m-k-1)}\{-k,k+1\} 
\end{equation}
and
\begin{align*}
  \frac{|b_k|^2}{\phi_{k\Im}^2}
=&\(\frac{2|\th_k|}{1+|\th_k|^2}\)^2
\lesssim \(\(2\nu^k\si^{(m-k)}\{-k,k\}\)\(2\si^{(m-k)}\{0,0\}\)\)^2
\\
\lesssim &
 \(4\ka_1\nu^k\si^{(m-k)}\{-k,k\}\)^2
\lesssim 
 16\ka_1^3 \(\frac{2\nu^2 \mu}{\rho}\)^k \si^{(m-k)}\{-k,k\}
\\
\lesssim &
  \frac{32\ka_1^3 \nu^2 \mu}{\rho} \si^{(m-k)}\{-k,k\}
\lesssim 
  \si^{(m-k)}\{-k,k\}
\\
\lesssim &
  \(\frac{\mu}{\rho}\)^k \si^{(m-k)}\{0,0\}
\end{align*}
for $\rho \ge \tilde{\rho}$. 
Let us define $\al_k$ and $\be_k$ by 
\begin{equation*}
  \al_k:=-1+\sqrt{1-\frac{|b_k|^2}{\phi_{k\Im}^2}}
  \; \text{ and } \;
  \be_k:=-\frac{\Im\{\ol{\th_k}b_{k+1}\}}{\phi_{1\Im}} 
\end{equation*}
for $k=1,\ldots,m-1$. 
By (\ref{phi_k+1-R}) and (\ref{phi_k+1-I}), 
we have
\begin{equation}\label{phi_k+1-R-rep}
  \phi_{(k+1)\Re}
  =\pa_t \log\sqrt{\frac{a}{\prod_{j=1}^k(1-|\th_j|^2)}}
\end{equation}
and
\begin{align*}
  \phi_{(k+1)\Im}
  =&\phi_{k\Im}(1+\al_k)+\phi_{1\Im}\be_k
\\
  =&\phi_{(k-1)\Im}(1+\al_k)(1+\al_{k-1})
   +\phi_{1\Im}\( \be_{k-1}(1+\al_k) + \be_k\)
\\
  &\vdots
\\
 =&\phi_{1\Im}
  \(\prod_{l=1}^k(1+\al_l) 
    + \sum_{j=1}^{k-1}\be_j\prod_{l=j+1}^k(1+\al_l)+\be_k\). 
\end{align*}
By Lemma \ref{symbol-p3} (v), we have
\begin{align*}
  \al_k 
  &\lesssim
  \(\frac{\mu}{\rho}\)^k \si^{(m-k)}\{0,0\}
    \(\frac12+\(\frac{\mu}{\rho}\)^k\si^{(m-k)}\{0,0\}\)
\\  
  &\lesssim
  \frac12\(\frac{\mu}{\rho}\)^k \si^{(m-k)}\{0,0\}
 +\ka_1 \(\frac{\mu}{\rho}\)^{2k}\si^{(m-k)}\{0,0\}
\\
  &\lesssim
  \(\frac{\mu}{\rho}\)^k \si^{(m-k)}\{0,0\}
\end{align*}
and
\begin{align*}
  \be_k
 &\lesssim 
  \(\nu^k\si^{(m-k)}\{-k,k\}\)\(2\ka_1 \nu^k \si^{(m-k-1)}\{-k,k+1\}\)
  \(\si^{(m)}\{-1,0\}\)
\\
 &\lesssim 
  2\ka_1^2 \nu^{2k}
  \si^{(m-k)}\{-k,k\} \si^{(m-k-1)}\{-k-1,k+1\}
\\
 &\lesssim 
  \frac{2\ka_1^3\mu}{\rho}
  \(\frac{\nu^2 \mu}{\rho}\)^{k} \(\frac{\mu}{\rho}\)^{k}
  \si^{(m-k-1)}\{0,0\}
\lesssim 
  \(\frac{\mu}{\rho}\)^{k} \si^{(m-k-1)}\{0,0\}. 
\end{align*}
Therefore, denoting 
$\psi_k:=\phi_{(k+1)\Im}/\phi_{1\Im}$, 
$N:=\rho/\mu$ and $\si_0:=\si^{(m-k-1)}\{0,0\}$, we have 
\begin{align*}
  \psi_k
  &\lesssim
  \prod_{l=1}^k\(1+N^{-l}\si_0\) 
  + \sum_{j=1}^{k-1} N^{-j}\si_0 \prod_{l=j+1}^k\(1+ N^{-l}\si_0\)+N^{-k}\si_0
\\
  &\lesssim
  \(1+ \sum_{j=1}^{k} N^{-j}\si_0\) \prod_{l=1}^k\(1+ N^{-l}\si_0\)
\\
  &\lesssim
  \(1+\sum_{j=1}^k N^{-j} \si_0\)
  \(1+\sum_{j=1}^\infty\(\ka_1 N^{-1}\)^j \si_0 \)^2
  \lesssim
  \(1+\frac{\ka_1N^{-1}}{1-\ka_1N^{-1}}\si_0\)^3
\\
  &\lesssim
  \(1+2\ka_1N^{-1}\si_0\)^3
  =1+6\ka_1N^{-1}\si_0+12\ka_1^2N^{-2}\si_0^2+8\ka_1^3N^{-3}\si_0^3
\\
  &\lesssim
  1+2N^{-1}\(3\ka_1+6\ka_1^3+4\ka_1^5\)\si_0
  \lesssim
  2\si_0
\end{align*}
for $N=\rho/\mu\ge 6\ka_1+12\ka_1^3+8\ka_1^5$, where we used the inequality 
$\prod^k_{l=1}(1+N^{-1})\le (1+\sum_{j=1}^\infty N^{-j})^2$. 
Consequently, we have 
\begin{equation}\label{phi_k+1/phi1}
  \frac{\phi_{1\Im}}{\phi_{(k+1)\Im}} \lesssim 2\si^{(m-k-1)}\{0,0\}.
\end{equation}
By (\ref{symbol_1/phi1}), (\ref{b_k+1-1}) and (\ref{phi_k+1/phi1}), we have 
\begin{align*}
  \frac{b_{k+1}}{\phi_{(k+1)\Im}}
 =&
  \frac{|b_{k+1}|}{\phi_{1\Im}}\frac{\phi_{1\Im}}{\phi_{(k+1)\Im}}
  \lesssim
  4\ka_1^3 \nu^k \si^{(m-k-1)}\{-k-1,k+1\}
\end{align*}
and
\begin{align*}
  \frac{|b_{k+1}|^2}{\phi_{(k+1)\Im}^2}
  \lesssim&
  16\ka_1^7 \nu^{2k} \(\frac{2\mu}{\rho}\)^{k+1} \si^{(m-k-1)}\{-k-1,k+1\}
\\
  \lesssim& 
  16\ka_1^7 \nu^{2k} \(\frac{2\mu^2}{\rho^2}\)^{k+1} \si^{(m-k-1)}\{0,0\}
  \lesssim 
  \frac12\si^{(m-k-1)}\{0,0\}. 
\end{align*}
Therefore, by Lemma \ref{symbol-p3} (v), we obtain 
\begin{align*}
  \th_{k+1}
 =&\frac{-\frac{\i b_{k+1}}{\phi_{(k+1)\Im}}
    \(1-\sqrt{1-\frac{|b_{k+1}|^2}{\phi_{(k+1)\Im}^2}}\)}
    {\frac{|b_{k+1}|^2}{\phi_{(k+1)\Im}^2}}
\\
 \lesssim&
  4\ka_1^3 \nu^k \si^{(m-k-1)}\{-k-1,k+1\}
  \(\frac12+ \frac12\si^{(m-k-1)}\{0,0\}\)
\\
 \lesssim&
  4\ka_1^4 \nu^k \si^{(m-k-1)}\{-k-1,k+1\}. 
\end{align*}
Consequently, the estimates (\ref{est-bk-th_k-e1}) and (\ref{est-bk-th_k-e2}) are established for any $k=1,\cdots,m$. 
\end{proof}

\subsection{Conclusion of the proof of Proposition \ref{PropL}}
We restrict ourselves $T_0=0$ without loss of generality. 
The estimate (\ref{est-bk-th_k-e2}) gives 
$|\th_k|\le (\nu\mu/\rho)^k \le 4^{-k}$, 
hence (\ref{V1}) is reduced to (\ref{Vm}) in $Z_{H}(m,\rho)$ by Lemma \ref{est-bk-th_k}. 
By (\ref{Vm}) and the estimate 
$|b_m|\le (\nu \mu)^m M_m |\xi|^{-m+1}$, 
which follows from (\ref{est-bk-th_k-e1}) with $k=m$, we have
\begin{align*}
  \pa_t|V_m(t,\xi)|^2
  \le& 2\(\phi_{m\Re}(t,\xi)+|b_m(t,\xi)|\)|V_m(t,\xi)|^2
\\
  \le& 2\(\phi_{m\Re}(t,\xi)+ (\nu\mu)^m M_m |\xi|^{-m+1}\)
  |V_m(t,\xi)|^2
\end{align*}
in $Z_{H}(m,\rho)$ uniformly on $[0,T)$. 
By Gronwall's inequality and (\ref{phi_k+1-R-rep}) with $k=m-1$, we have 
\begin{equation}\label{dVm}
  |V_m(t,\xi)|^2 
  \le 
    \frac{a(t)}{a(0)}
    \(\prod_{k=1}^{m-1}
    \frac{1-|\th_k(0,\xi)|^2}{1-|\th_k(t,\xi)|^2}
    \)
  \exp\(
  2T(\nu\mu)^m M_m |\xi|^{-m+1}
\)|V_m(0,\xi)|^2. 
\end{equation}
Here we note the following inequalities are established: 
\begin{align*}
  |V_{k+1}|^2
 &=|\Th_k^{-1}V_k|^2
 =\frac{1}{\(1-|\th_k|^2\)^2}
  \(\(1+|\th_k|^2\)|V_k|^2-4\Re\{\th_k(V_{k})_1\ol{(V_{k})_2}\}\)
\\
 &\begin{cases}
   \le\dfrac{\(1+|\th_k|\)^2}{\(1-|\th_k|^2\)^2}|V_k|^2
     =\dfrac{1}{\(1-|\th_k|\)^2}|V_k|^2,\\[3mm]
   \ge\dfrac{\(1-|\th_k|\)^2}{\(1-|\th_k|^2\)^2}|V_k|^2
     =\dfrac{1}{\(1+|\th_k|\)^2}|V_k|^2,
  \end{cases}
\end{align*}
\begin{align*}
  \frac{1}{\prod_{k=1}^{m-1}\(1+|\th_k|\)^2}|V_1|^2
  \le |V_m|^2 \le
  \frac{1}{\prod_{k=1}^{m-1}\(1-|\th_k|\)^2}|V_1|^2,
\end{align*}
\begin{align*}
\prod_{k=1}^{m-1} 
  \frac{\(1+|\th_k(0,\xi)|\)\(1+|\th_k(t,\xi)|\)}
       {\(1-|\th_k(0,\xi)|\)\(1-|\th_k(t,\xi)|\)}
 \le \(\prod_{k=1}^{m-1} \frac{1+4^{-k}}{1-4^{-k}}\)^2
 \le \(\frac{1+\sum_{k=1}^\infty 4^{-k}}{1-\sum_{k=1}^\infty 4^{-k}}\)^4
 =81,
\end{align*}
and
\begin{equation*}
  4\cE(t,\xi)\le |V_1(t,\xi)|^2 \le 4a_1^2 \cE(t,\xi). 
\end{equation*}
Then, by (\ref{dVm}) we have 
\begin{align*}
  \cE(t,\xi) 
  \le& \frac14|V_1(t,\xi)|^2
  \le \frac{\prod_{k=1}^{m-1}\(1+|\th_k(t,\xi)|\)^2}{4}|V_m(t,\xi)|^2
\\
  \le&
  \frac{a(t)\prod_{k=1}^{m-1}\(1+|\th_k(t,\xi)|\)^2}{4a(0)}
  \(\prod_{k=1}^{m-1} \frac{1-|\th_k(t,\xi)|^2}{1-|\th_k(t,\xi)|^2}\)
  \exp\(2T(\nu\mu)^m M_m |\xi|^{-m+1}\)
  |V_m(0,\xi)|^2
\\
  \le&
  \frac{a(t)}{4a(0)}
  \(\prod_{k=1}^{m-1} \frac{\(1+|\th_k(t,\xi)|\)^2}{\(1-|\th_k(0,\xi)|\)^2}\)
  \(\prod_{k=1}^{m-1} \frac{1-|\th_k(0,\xi)|^2}{1-|\th_k(t,\xi)|^2}\)
  \exp\(2T(\nu\mu)^m M_m |\xi|^{-m+1}\)
  |V_1(0,\xi)|^2
\\
  \le&
  81a_1^3
  \exp\(2T(\nu\mu)^m M_m |\xi|^{-m+1}\)
  \cE(0,\xi)
\end{align*}
for any $t\in(0,T)$ in $Z_{H}(m,\rho)$. 
Therefore, setting $\ka_0=8e^2\ka_1^5$, and noting 
$\nu\mu = 4 e^2 \mu_0 \ka_1^5 = \mu_0 \ka_0/2$, 
we have the estimate (\ref{PropL-e2}) with $C_0=81a_1^3$. 

\section{Proof of the main theorem}

\subsection{Energy conservation}
Let us denote $v(t,\xi)=\hat{u}(t,\xi)$, and define 
$\Phi(t)=\Phi(t;v)$ by 
\begin{equation*}
  \Phi(t;v)=1+\|\nabla v(t,\cd)\|^2. 
\end{equation*}
Then (\ref{K}) is reduced to the following problem: 
\begin{equation}\label{K0}
\begin{cases}
  \pa_t^2 v(t,\xi)+\Phi(t)|\xi|^2 v(t,\xi)=0,
  \quad
  (t,\xi)\in(0,\infty)\times \R^n,
\\
  v(0,\xi)=\hat{u}_0(\xi),
  \quad
  (\pa_t v)(0,\xi)=\hat{u}_1(\xi),
  \quad
  \xi\in \R^n. 
\end{cases}
\end{equation}
Let $\cE(t,\xi)$ be the energy functional defined by (\ref{cE}) to the solution of (\ref{K0}), and $E_0(t)=E_0(t;v)$ the total energy to the solution of (\ref{K0}) at $t$ defined by 
\begin{equation}\label{E0}
  E_0(t):=
  \frac12
  \(\|\pa_t v(t,\cd)\|^2
  +\int^{\||\cd|v(t,\cd)\|^2}_0
     \(1+y\)\,dy \).
\end{equation}
Then we have following properties of the energy conservation and corresponding estimates: 
\begin{Lemm}\label{EC}
If a strong solution of Kirchhoff equation $(\ref{K})$ exists on $[0,T)$, 
then the following estimates are established for any $t\in[0,T):$ 
\begin{equation}\label{E0cE}
  \int_{\R^n}\cE(t,\xi)\,d\xi \le E_0(t) = E_0(0)
\end{equation}
and
\begin{equation}\label{Phi0}
  1 \le \Phi(t) \le 1 + 2E_0(0). 
\end{equation}
\end{Lemm}
\begin{proof}
The equality of (\ref{E0cE}) is straightforward by multiplying $\pa_t v(t,\xi)$ to the equation of (\ref{K0}) and integrating over $\R^n_\xi$. 
The inequalities of (\ref{E0cE}) and (\ref{Phi0}) are trivial by the definition of $E_0(t)$. 
\end{proof}
If $E_0(0)=0$ then (\ref{K}) has only a trivial solution. 
Therefore, we can suppose that $E_0(0)>0$ without loss of generality.

\subsection{Estimates of the higher order derivatives of $\Phi(t)$}
Let us estimate the higher order derivatives of $\Phi(t)$ in order to apply Proposition \ref{PropL} for $a(t)=\sqrt{\Phi(t)}$. 
Suppose that 
\begin{equation*}
  \int_{\R^n}|\xi|\cE(t,\xi)\,d\xi<\infty
\end{equation*}
for any $t\in [0,T)$. 
Then we have 
\begin{equation}\label{Phi'bdd}
  |\Phi'(t)|
 =\left|2\Re\int_{\R^n}|\xi|^2v_t(t,\xi)\ol{v(t,\xi)}\,d\xi\right|
  \le 2\int_{\R^n}|\xi|\cE(t,\xi)\,d\xi
  <\infty
\end{equation}
for $t\in [0,T)$. 
Noting the estimates (\ref{Phi0}), 
\begin{align*}
 &\pa_t \(\cE(t,\xi)+\frac12(\Phi(t)-1)|\xi|^2|v(t,\xi)|^2\)
\\
 &=\frac12\Phi'(t)|\xi|^2|v(t,\xi)|^2
 \le |\Phi'(t)|\(\cE(t,\xi)+\frac12(\Phi(t)-1)|\xi|^2|v(t,\xi)|^2\)
\end{align*}
and Gronwall's inequality, we have 
\begin{align*}
  \cE(t,\xi)+\frac12(\Phi(t)-1)|\xi|^2|v(t,\xi)|^2
  \le
     \exp\(\int^t_0|\Phi'(s)|\,ds\)
     \(\cE(0,\xi)+\frac12(\Phi(0)-1)|\xi|^2|v(0,\xi)|^2\). 
\end{align*}
It follows that 
\begin{equation}\label{estcE}
  \cE(t,\xi) \le
  (1+2E_0(0))\exp\(\int^t_0|\Phi'(s)|\,ds\)\cE(0,\xi). 
\end{equation}
If $\int_{\R^n}|\xi|^2\cE(0,\xi)\,d\xi <\infty$, 
then by (\ref{estcE}) we have 
\begin{align*}
  |\Phi''(t)|
  =&\left|2\Re\int_{\R^n}|\xi|^2\(-\Phi(t)|\xi|^2|v(t,\xi)|^2
  +|v_t(t,\xi)|^2\)\,d\xi\right|
\\
  \le &
  2(1+2E_0(0))\int_{\R^n}|\xi|^2\cE(t,\xi)\,d\xi
\\
  \le &
  2(1+2E_0(0))^2 \exp\(\int^t_0|\Phi'(s)|\,ds\)
  \int_{\R^n}|\xi|^2\cE(0,\xi)\,d\xi
\\
  <& \infty
\end{align*}
for any $t\in [0,T)$. 
By the same way, if $\int_{\R^n}|\xi|^k\cE(0,\xi)\,d\xi <\infty$ for any $k\in \N$, then we have $|\Phi^{(k)}(t)|<\infty$ on $[0,T)$. 
Precisely, we have the following lemmas for the higher derivatives of $\Phi(t)$: 
\begin{Lemm}\label{Phik}
We define $P_{l,j}=P_{l,j}(t;v)$ for $l=0,1,2$ and $j\in \N_0$ as follows$:$ 
\begin{align*}
  &P_{0,j}=\int_{\R^n}|\xi|^j |v_t(t,\xi)|^2\,d\xi, \\
  &P_{1,j}=\Re\int_{\R^n}|\xi|^{j+1} v_t(t,\xi)\ol{v(t,\xi)}\,d\xi, \\
  &P_{2,j}=\int_{\R^n}|\xi|^{j+2} |v(t,\xi)|^2\,d\xi. 
\end{align*}
Then for $k\in \N$ we have the following representations$:$ 
\begin{equation}\label{LemmPhik-e}
  \Phi^{(k)}(t) 
 =\sum_{(\al,\be,\ga)\in \La_k}
  I(\al,\be,\ga) \prod_{j=0}^k 
  P_{0,j}^{\al_j}P_{1,j}^{\be_j}P_{2,j}^{\ga_j},
\end{equation}
where 
\begin{equation}\label{LemmPhik-e1}
  \La_k:=\left\{(\al,\be,\ga)\in \(\N_0^{k+1}\)^3\:;\:
  \sum_{j=1}^k j(\al_j+\be_j+\ga_j)=k,\;\;
  \sum_{j=0}^k (\al_j+\be_j+\ga_j) 
  \le \left[\frac{k+2}{2}\right]
  \right\}
\end{equation}
and $I(\al,\be,\ga)$ are integers satisfying 
\begin{equation}\label{LemmPhik-e2}
  \sum_{(\al,\be,\ga)\in \La_k}|I(\al,\be,\ga)|
  \le 3^k k!. 
\end{equation}
\end{Lemm}
\begin{proof}
Noting $\Phi^{(1)}(t)=2P_{1,1}$, (\ref{LemmPhik-e}) is valid for $k=1$. 
Let us suppose that (\ref{LemmPhik-e}) is valid for a $k(\ge 1)$. 
Noting the equalities 
\begin{align*}
  \frac{d}{dt}P_{0,l}
 =-2\Phi(t)\Re \int|\xi|^{l+2} v_t \ol{v}\,d\xi
 =-2\(1+P_{2,0}\)P_{1,l+1},
\end{align*}
\begin{align*}
  \frac{d}{dt}P_{1,l}
=-\(1+P_{2,0}\)\Re\int|\xi|^{l+3} |v|^2\,d\xi + P_{0,l+1}
=-\(1+P_{2,0}\)P_{2,l+1} + P_{0,l+1}
\end{align*}
and
\begin{align*}
  \frac{d}{dt}P_{2,l}
 =&2\Re \int|\xi|^{l+2} v_t\ol{v}\,d\xi
 =2P_{1,l+1},
\end{align*}
we have 
\begin{align*}
\frac{d}{dt}& \prod_{j=0}^k 
  P_{0,j}^{\al_j}P_{1,j}^{\be_j}P_{2,j}^{\ga_j}
\\
=&\sum_{{0\le l \le k}\atop{\al_l\not=0}} 
  \(
 -2\al_l (1+P_{2,0}) P_{1,l+1} 
  P_{0,l}^{\al_l-1}
  \)
  \(\prod_{{0\le j \le k}\atop{j\not=l}}P_{0,j}^{\al_j}\)
  \(\prod_{{0\le j \le k}} P_{1,j}^{\be_j}P_{2,j}^{\ga_j}\)
\\
&+\sum_{{0\le l \le k}\atop{\be_l\not=0}}
  \(
  \be_l\(-(1+P_{2,0}) P_{2,l+1} + P_{0,l+1}\)
  P_{1,l}^{\be_l-1}
  \)
  \(\prod_{{0\le j \le k}\atop{j\not=l}}P_{1,j}^{\be_j}\)
  \(\prod_{{0\le j \le k}} P_{0,j}^{\al_j}P_{2,j}^{\ga_j}\)
\\
&+\sum_{{0\le l \le k}\atop{\ga_l\not=0}}
  2\ga_l P_{1,l+1}
  P_{2,l}^{\ga_l-1}
  \(\prod_{{0\le j \le k}\atop{j\not=l}}P_{2,j}^{\ga_j}\)
  \(\prod_{{0\le j \le k}} P_{0,j}^{\al_j}P_{1,j}^{\be_j}\). 
\end{align*}
It follows that the representation (\ref{LemmPhik-e}) 
is established for $k+1$. 
Moreover, if (\ref{LemmPhik-e2}) is valid, then we have
\begin{align*}
  \sum_{(\al,\be,\ga)\in \La_{k+1}}
  |I(\al,\be,\ga)|
  &\le
  \sum_{l=0}^k(4\al_l+3\be_l+2\ga_l)
  \sum_{(\al,\be,\ga)\in \La_{k}} |I(\al,\be,\ga)|
\\
  &\le
  4 \left[\frac{k+2}{2}\right] 3^k k!
  \le
    3^{k+1}(k+1)!,
\end{align*}
which concludes the proof of the lemma. 
\end{proof}
\begin{Lemm}\label{estPhik}
There exists a positive constant $\nu_0$ such that 
\begin{equation*}
  \left|\Phi^{(k)}(t)\right|
  \le \nu_0^k k! \max_{{h\in \N_0^k}\atop{|h|=k}}
  \left\{\prod_{l=1}^k \int_{\R^n}|\xi|^{h_l}\cE(t,\xi)\,d\xi\right\}
\end{equation*}
for any $k\in \N$, 
where $h=(h_1,\dots,h_k)\in \N_0^k$ and $|h|=\sum_{l=1}^k h_l$. 
\end{Lemm}
\begin{proof}
Noting $P_{0,j}+P_{2,j}=2\int_{\R^n}|\xi|^{j}\cE(t,\xi)\,d\xi$ 
and $|P_{1,j}|\le \int_{\R^n}|\xi|^{j}\cE(t,\xi)\,d\xi$, 
by (\ref{E0cE}) and Lemma \ref{Phik} we have 
\begin{align*}
  \left|\Phi^{(k)}(t)\right|
\le &
  2^{\sum_{j=0}^{k}\(\al_j+\be_j+\ga_j\)}
  E_0(t)^{\al_0+\be_0+\ga_0}
  \sum_{(\al,\be,\ga)\in \La_k}
  \left|I(\al,\be,\ga)\right| 
  \prod_{j=1}^k 
  \(\int_{\R^n}|\xi|^j \cE(t,\xi)\,d\xi\)^{\al_j+\be_j+\ga_j}
\\
\le &
  \(6 \max\{1,E_0(0)\} \)^k k! 
  \max_{(\al,\be,\ga)\in \La_k}
  \left\{
  \prod_{j=1}^k 
  \(\int_{\R^n}|\xi|^j 
  \cE(t,\xi)\,d\xi\)^{\al_j+\be_j+\ga_j}
  \right\}. 
\end{align*}
Therefore, noting (\ref{LemmPhik-e1}) we conclude the proof for 
$\nu_0 \ge 6 \max\{1,E_0(0)\}$. 
\end{proof}

\subsection{Prolongation of the local solution}
The existence of a unique local solution in Sobolev space is known. 
Hence we can define $T$ by 
\begin{equation*}
  T:=\sup\left\{\tau>0\:;\:
  \int_{\R^n} |\xi|\cE(t,\xi)\,d\xi<\infty,\;\; 
  \forall t\in [0,\tau] \right\}. 
\end{equation*}
Here (\ref{Phi'bdd}), (\ref{estcE}) and Remark \ref{rem-strongsol} imply that $T$ is the supremum of the existence time of the strong solution to (\ref{K}). 
Let us denote 
\begin{equation*}
  M_k:=k! L_k,
  \;\;
  \cM(r):=\mM\(r;\{M_k\}\),
  \;\;
  \tcM(r):=\mM\(r;\{L_k\}\). 
\end{equation*}
Then we shall prove $T=\infty$ by contradiction if there exist $\eta>0$, 
$\{\rho_j\}\in {\cal L}$ and $K_0>0$ such that 
\begin{equation*}
  \sup_{j}\left\{\int_{|\xi|\ge\rho_j}
  \tcM\(\frac{|\xi|}{\rho_j}\)
    \exp\(\frac{\eta|\xi|}{\cM\(\frac{|\xi|}{\rho_j}\)}\)
    \cE(0,\xi)\,d\xi\right\}
  \le K_0. 
\end{equation*}

Suppose that $T<\infty$, and define $T_0$ by 
\begin{equation*}
  T_0:=\max\{T-\eta,0\}. 
\end{equation*}
Then there exists a positive constant $C=C(T,\eta)$ such that 
\begin{equation}\label{cEbddT0}
  \max_{t\in[0,T_0]}\left\{\int_{\R^n}|\xi|\cE(t,\xi)\,d\xi\right\}\le C. 
\end{equation}
Moreover, we have the following lemma: 
\begin{Lemm}\label{lemm-est_<T_0}
The following estimate is established$:$ 
\begin{equation}\label{lemm-est_<T_0-e1}
  \sup_{j}\sup_{l\ge 1}
  \left\{\int_{|\xi|\ge\rho_j}
  \frac{1}{L_l}\(\frac{|\xi|}{\rho_j}\)^l
  \cE(t,\xi)\,d\xi\right\}
  \le (1+2E_0(0)) e^{2CT_0} K_0 
\end{equation}
for any $t\in[0,T_0]$. 
Moreover, we have 
\begin{equation}\label{lemm-est_<T_0-e2}
  \lim_{j\to\infty}
  \max_{t\in[0,T_0]}
  \sup_{l\ge 1} \left\{\int_{|\xi|\ge\rho_j}
  \frac{1}{L_l}\(\frac{|\xi|}{\rho_j}\)^l
  \cE(t,\xi)\,d\xi\right\}=0. 
\end{equation}
\end{Lemm}
\begin{proof}
By (\ref{Phi'bdd}), (\ref{estcE}) and (\ref{cEbddT0}) we have 
\begin{align*}
  \frac{1}{L_l}\(\frac{|\xi|}{\rho_j}\)^l \cE(t,\xi)
\le&
  (1+2E_0(0))e^{2CT_0} 
  \frac{1}{L_l}\(\frac{|\xi|}{\rho_j}\)^l \cE(0,\xi)
\\
\le&
  (1+2E_0(0))e^{2CT_0} 
  \tcM\(\frac{|\xi|}{\rho_j}\) 
  \exp\(\frac{\eta|\xi|}{\cM\(\frac{|\xi|}{\rho_j}\)}\)\cE(0,\xi)
\end{align*}
for any $j\in \N$. 
Integrating over $|\xi|\ge \rho_j$ we have (\ref{lemm-est_<T_0-e1}). 
Moreover, by (\ref{lemm-est_<T_0-e1}) we have 
\begin{align*}
  \int_{|\xi|\ge \rho_1} \frac{1}{L_l}\(\frac{|\xi|}{\rho_1}\)^l
  \cE(t,\xi)\,d\xi
 =&
  \sup_{j}
  \left\{\int_{|\xi|\ge \rho_j}
  \frac{1}{L_l}\(\frac{|\xi|}{\rho_j}\)^l
  \cE(t,\xi)\,d\xi\right\}
\\
 \le& (1+2E_0(0)) e^{2CT_0} K_0,
\end{align*}
it follows that 
\begin{align*}
  \int_{|\xi|\ge\rho_j} \frac{1}{L_l}\(\frac{|\xi|}{\rho_j}\)^l
  \cE(t,\xi)\,d\xi
  \le&
  \int_{|\xi|\ge\rho_j} \frac{1}{L_l}\(\frac{|\xi|}{\rho_1}\)^l
  \cE(t,\xi)\,d\xi 
  \to 0
  \;\;
  (j \to\infty)
\end{align*}
uniformly with respect to $t\in[0,T_0]$ and $l\ge 1$. 
\end{proof}

By (\ref{lemm-est_<T_0-e2}) there exists $j_0\in \N$ such that 
\begin{equation*}
  \max_{t\in[0,T_0]}
  \sup_{l\ge 1} \left\{\int_{|\xi|\ge\rho_j}
  \frac{1}{L_l}\(\frac{|\xi|}{\rho_j}\)^l
  \cE(t,\xi)\,d\xi\right\}
  \le \frac{E_0(0)}{2}
\end{equation*}
for any $j\ge j_0$. 
Moreover, by the continuity of $\cE(t,\xi)$, there exists $T_1\in(T_0,T)$ 
such that 
\begin{equation*}
  \sup_{t\in[0,T_1)}
  \sup_{l\ge 1} \left\{\int_{|\xi|\ge\rho_j}
  \frac{1}{L_l}\(\frac{|\xi|}{\rho_j}\)^l
  \cE(t,\xi)\,d\xi\right\}
  <E_0(0)
\end{equation*}
for any $j\ge j_0$. 
We shall prove that there exists $j_1 \ge j_0$ and $\si_1\ge 1$ such that we have a contradiction if there exists $T_1\in (T_0,T)$ such that 
\begin{equation}\label{maxT1=E0}
  \sup_{t\in[0,T_1)}
  \sup_{l\ge 1} \left\{\int_{|\xi|\ge\rho_{j_1}}
  \frac{1}{L_l}\(\frac{|\xi|}{\si_1 \rho_{j_1}}\)^l
  \cE(t,\xi)\,d\xi\right\}
  =E_0(0), 
\end{equation}
which implies $T=T_1$. 
However, (\ref{maxT1=E0}) with $T_1=T$ implies $\sup_{t\in[0,T)}\int_{\R^n}|\xi|\cE(t,\xi)\,d\xi<\infty$, and this estimate contradicts $T<\infty$. 

Suppose that there exists $T_1\in(T_0,T)$ such that 
(\ref{maxT1=E0}) is valid. 
Noting (\ref{E0cE}), for any $t\in[0,T_1]$ and $l\ge 1$ we have 
\begin{align*}
  \int_{\R^n}|\xi|^l \cE(t,\xi)\,d\xi
=&\(\si_1 \rho_{j_1}\)^l L_l 
  \int_{|\xi|\ge \rho_{j_1}} \frac{1}{L_l}
  \(\frac{|\xi|}{\si_1\rho_{j_1}}\)^l 
  \cE(t,\xi)\,d\xi
  + \int_{|\xi|\le \rho_{j_1}}|\xi|^l \cE(t,\xi)\,d\xi
\\
\le&
  E_0(0)\(\si_1 \rho_{j_1}\)^l L_l
  + \rho_{j_1}^l \int_{|\xi|\le \rho_{j_1}} \cE(t,\xi)\,d\xi
\\
\le&
  2E_0(0)\(\si_1 \rho_{j_1}\)^l L_l
\end{align*}
for any $j_1\ge j_0$ and $\si_1\ge 1$. 
By Lemma \ref{estPhik} we have 
\begin{align*}
  \left|\Phi^{(k)}(t)\right|
  \le &
  \nu_0^k k! \max_{{h\in \N_0^k}\atop{|h|=k}}
    \left\{\prod_{l=1}^k 2E_0(0)\(\si_1 \rho_{j_1}\)^{h_l} L_{h_l} \right\}
\\
  \le&
  \(2E_0(0)\si_1 \rho_{j_1} \nu_0\)^k k! \max_{{h\in \N_0^k}\atop{|h|=k}}
    \left\{\prod_{l=1}^k \frac{M_{h_l}}{h_l!} \right\} 
  \le
  \nu_1^k M_k
\end{align*}
for any $t\in[0,T_1)$ and $k\in \N$, 
where $\nu_1=2E_0(0)\si_1 \rho_{j_1} \nu_0$, and we used the following inequality: 
\begin{equation*}
  \prod_{l=1}^k \frac{M_{h_l}}{h_l!}
  \le \frac{M_{h_1+\cdots+h_k}}{(h_1+\cdots+h_k)!}
  =\frac{M_k}{k!},
\end{equation*}
which follows from (\ref{lc-binomial-e1}). 
Moreover, by Lemma \ref{Lemm_sqrt_f} in Appendix, there exists a positive constant $\mu_0$ such that 
\begin{equation*}
  \left|\frac{d^k}{dt^k}\sqrt{\Phi(t)}\right|
  \le \mu_0^k M_k
\end{equation*}
for any $k\ge 1$; 
thus we can apply Proposition \ref{PropL} for $a(t)=\sqrt{\Phi(t)}$ 
and $a_1=\sqrt{1+2E_0(0)}$ due to (\ref{Phi0}).

By applying Proposition \ref{PropL} with $k=2$ for 
$t\in[0,T_0]$ and $|\xi|\ge \ka\mu_0 M_2$, we have 
\begin{align*}
  \cE(t,\xi)\le C_0\cE(0,\xi)\exp\(T_0 (\ka\mu_0)^2 M_2 |\xi|^{-1}\). 
\end{align*}
If we choose $j_1$ by 
\begin{equation*}
  j_1:=\min\left\{j\ge j_0\:;\: 
  \rho_j\ge \max\{\ka\mu_0 M_2,T_0 (\ka\mu_0)^2M_2/\log 2\}\right\}, 
\end{equation*}
then we have 
\begin{equation}\label{estcET0}
  \cE(t,\xi)\le 2C_0\cE(0,\xi)
\end{equation}
for any  $t\in[0,T_0]$ and $|\xi|\ge \rho_{j_1}$. 
Moreover, by Proposition \ref{PropL} we have 
\begin{equation}\label{estcET1}
  \cE(t,\xi) \le C_0\exp\((T-T_0)\rho_{j_1}^k M_k |\xi|^{-k+1}\)\cE(T_0,\xi) 
\end{equation}
for $t\in(T_0,T_1)$ and $|\xi|\ge \rho_{j_1} M_k/M_{k-1}$. 
Consequently, associating the estimates (\ref{estcET0}) and (\ref{estcET1}) 
we have 
\begin{equation*}
  \cE(t,\xi) \le 2C_0^2\exp\((T-T_0)\rho_{j_1}^k M_k |\xi|^{-k+1}\)\cE(0,\xi)
\end{equation*}
for any 
$t\in [0,T_1)$ and $|\xi|\ge \rho_{j_1} M_k/M_{k-1}$. 
Denoting $I_k:=[M_k/M_{k-1},M_{k+1}/M_k)$ and using (\ref{lemma-af-e2}) we have \begin{align*}
 \int_{|\xi|\ge \rho_{j_1}}
    \tcM\(\frac{|\xi|}{\rho_{j_1}}\)\cE(t,\xi)\,d\xi
&=\sum_{k=1}^\infty
  \int_{\frac{|\xi|}{\rho_{j_1}}\in I_k}
    \tcM\(\frac{|\xi|}{\rho_{j_1}}\)\cE(t,\xi)\,d\xi
\\
&\le
  2C_0^2\sum_{k=1}^\infty \int_{\frac{|\xi|}{\rho_1}\in I_k}
    \tcM\(\frac{|\xi|}{\rho_{j_1}}\)
    \exp\((T-T_0) \rho_{j_1}^k M_k |\xi|^{-k+1}\)
    \cE(0,\xi)\,d\xi
\\
&\le
  2C_0^2\sum_{k=1}^\infty \int_{\frac{|\xi|}{\rho_{j_1}}\in I_k}
    \tcM\(\frac{|\xi|}{\rho_{j_1}}\)
    \exp\(\eta \rho_{j_1}^k M_k |\xi|^{-k+1}\)
    \cE(0,\xi)\,d\xi
\\
 &=
  2 C_0^2\int_{|\xi|\ge \rho_{j_1}}
    \tcM\(\frac{|\xi|}{\rho_{j_1}}\)
    \exp\(\frac{\eta|\xi|}{\cM\(\frac{\rho_{j_1}}{|\xi|}\)}\)
    \cE(0,\xi)\,d\xi
\\
&\le 2C_0^2K_0. 
\end{align*}
Therefore, seeting $\si_1=4C_0^2K_0/E_0(0)$ we have 
\begin{align*}
  \int_{|\xi|\ge \rho_{j_1}}
    \frac{1}{L_l}\(\frac{|\xi|}{\si_1\rho_{j_1}}\)^l \cE(t,\xi) \,d\xi
  &=\frac{1}{\si_1^l}
  \int_{|\xi|\ge \rho_{j_1}}
    \frac{1}{L_l}\(\frac{|\xi|}{\rho_{j_1}}\)^l \cE(t,\xi) \,d\xi
  \\
  &\le \frac{1}{\si_1^l}
  \int_{|\xi|\ge \rho_{j_1}}
    \tcM\(\frac{|\xi|}{\rho_{j_1}}\) \cE(t,\xi) \,d\xi
\\
  &\le \frac{2C_0^2 K_0}{\si_1}
  = \frac{E_0(0)}{2}
\end{align*}
for $t\in [0,T_1)$; 
however, these estimates contradict (\ref{maxT1=E0}). 
Consequently, it must be that 
\begin{equation*}
  \sup_{t\in[0,T)}
  \sup_{l\ge 1} \left\{\int_{|\xi|\ge\rho_{j_1}}
  \frac{1}{L_l}\(\frac{|\xi|}{\si_1 \rho_{j_1}}\)^l
  \cE(t,\xi)\,d\xi\right\}
  \le E_0(0). 
\end{equation*}
It follows that 
\begin{align*}
  \int_{\R^n}|\xi|\cE(t,\xi)\,d\xi
=&\int_{|\xi|\le \rho_{j_1}}|\xi|\cE(t,\xi)\,d\xi
 +\tM_1\si_1\rho_{j_1}
  \int_{|\xi|\ge \rho_{j_1}}\frac{|\xi|}{L_1\si_1\rho_{j_1}}\cE(t,\xi)\,d\xi
\\
\le&
  \rho_{j_1} E_0(0)
 +L_1\si_1\rho_{j_1} \sup_{l\ge 1}
  \left\{
    \int_{|\xi|\ge \rho_{j_1}}\frac{1}{L_l}\(\frac{|\xi|}{\si_1\rho_{j_1}}\)^l
    \cE(t,\xi)\,d\xi
  \right\}
\\
\le&
  \rho_{j_1} E_0(0)
 +L_1\si_1\rho_{j_1} E_0(0)
\end{align*}
for any $[0,T)$. 
However, these estimates bring a contradiction if $T<\infty$; 
thus it must be that $T=\infty$.

\subsection{Concluding remarks}
\begin{itemize}
\item 
The nonlinear coefficient $1+\|\nabla u(t,\cd)\|^2$ in the equation of (\ref{K}) can be generalized $\Psi(\|\nabla u(t,\cd)\|^2)$ for $\Psi\in C^\infty([0,\infty))$ satisfying $\Psi(y)\ge 1$ with the assumption $(u_0,u_1)\in \cB_\Delta(\{\widetilde{L}_k\},\{L_k\})$, where $\widetilde{L}_k$ will be chosen corresponding to the regularity of $\Psi$. 
For example, if $|\Psi^{(k)}(y)|\le \mu^k k!^{s}$ with $s>1$, then $\widetilde{L}_k=k!^{s+1}L_k$. 
\item 
We may expect that the assumption of Theorem \ref{MThm} to the initial data is improved to $\cB_\Delta(\{L_k\},\{L_k\})$. 
However, the right hand side of (\ref{LemmPhik-e2}) must be estimated by $K^k$ for a positive constant $K$ in order to realize the expectation, 
and such estimates are not proved at present. 
\end{itemize}

\section{Appendix}
\subsection{Proof of Lemma \ref{lc-prop}}
\begin{proof}
(i) is trivial by (\ref{lc}) as follows: 
\begin{align*}
  \frac{M_l}{M_{l-1}}
  \le \frac{l M_{l+1}}{(l+1)M_{l}}
  \le \frac{l M_{l+2}}{(l+2)M_{l}}
  \le \cdots \le \frac{lM_k}{kM_{k-1}}. 
\end{align*}
(ii) is also trivial by (i) as follows: 
\begin{align*}
  M_{k}M_{l}
 =M_{k+l}
  \frac{\frac{M_l}{M_{l-1}}\cdots\frac{M_1}{M_0}}
       {\frac{M_{k+l}}{M_{k+l-1}}\cdots \frac{M_{k+1}}{M_k}}
 \le M_{k+l}. 
\end{align*}
By (i) and Stirling's formula, there exists a constant $\de\in(0,1)$ 
such that 
\begin{align*}
  M_k=\frac{M_k}{M_0}
  =\frac{M_k}{M_{k-1}}\frac{M_{k-1}}{M_{k-2}}\cdots\frac{M_1}{M_0}
  \le
  \frac{k!}{k^k}\(\frac{M_k}{M_{k-1}}\)^k
  \le \de^k \(\frac{M_k}{M_{k-1}}\)^k. 
\end{align*}
Thus (iii) is proved. 
By (i) we have (iv) as follows: 
\begin{align*}
  \frac{M_k}{M_{k-1}}-\frac{M_l}{M_{l-1}}
  \ge
  \frac{kM_{l}}{lM_{l-1}}-\frac{M_l}{M_{l-1}}
  \ge \(\frac{k}{l}-1\)\frac{lM_1}{M_{0}}
  =k-l.
\end{align*}
By (i), (v) is trivial as follows: 
\begin{align*}
  \frac{k!M_k}{k (k-1)!M_{k-1}}
 =\frac{M_k}{M_{k-1}}
<  \frac{M_{k+1}}{M_{k}}
 =\frac{(k+1)!M_{k+1}}{(k+1)k!M_{k}}. 
\end{align*}
\end{proof}

\subsection{Proof of Lemma \ref{lc-binomial}}
\begin{proof}
By (\ref{lc}) we have 
\begin{align*}
  \binom{k}{j}\frac{M_{r+j} M_{q+k-j}}{M_{q+r+k}}
=&\binom{k}{j}\prod_{l=0}^{q+k-j-1}\frac{1}{r+j+l+1}
  \frac{(r+j+l+1)M_{r+j+l}}{M_{r+j+l+1}}
  \frac{M_{l+1}}{M_{l}}
\\
\le&\binom{k}{j}\prod_{l=0}^{q+k-j-1}\frac{1}{r+j+l+1}
  \frac{(l+1)M_{l}}{M_{l+1}}\frac{M_{l+1}}{M_{l}}
\\
\le&\binom{k}{j}\prod_{l=0}^{k-j-1}\frac{l+1}{j+l+1}
=1.
\end{align*}
We can suppose that $r\ge 1$; otherwise (\ref{lc-binomial-e2}) coincides with (\ref{lc-binomial-e1}). 
Let us define 
\begin{align*}
  L_j(k,r):=\frac{(k-j+r)!(j+r)!}{(k-j)!j!}
\end{align*}
and
\begin{align*}
  N_j(k,r):
 =\frac{\binom{k}{j}\prod_{l=1}^{j} \frac{r+l}{r+k-j+l}}{\binom{k+r-1}{r}}
 =\frac{(k-1)!k!}{(k+r-1)!(k+r)!}L_j(k,r).
\end{align*}
Noting that $L_{j-1}(k,r)\le L_{j}(k,r)$ if and only if $2j-1\le k$, 
we have
\begin{align*}
  \max_{0\le j \le k}\{L_j(k,r)\}=L_{\left[\frac{k+1}{2}\right]}(k,r).
\end{align*}
Consequently, if $k$ is odd then 
\begin{align*}
  \max_{0\le j \le k}\{N_j(k,r)\}
=\frac{(\frac{k-1}{2}+r)\cdots(\frac{k-1}{2}+1)}{(k+r-1)\cdots k}
 \frac{(\frac{k+1}{2}+r)\cdots(\frac{k+1}{2}+1)}{(k+r)\cdots(k+1)}
  \le 1,
\end{align*}
and if $k$ is even then
\begin{align*}
  \max_{0\le j \le k}\{N_j(k,r)\}
=&\frac{(\frac{k}{2}+r)\cdots(\frac{k}{2}+1)}{(k+r-1)\cdots k}
  \frac{(\frac{k}{2}+r)\cdots(\frac{k}{2}+1)}{(k+r)\cdots(k+1)}
  \le 1.
\end{align*}
Therefore, by (\ref{lc}) we have 
\begin{align*}
  \binom{k}{j}\frac{M_{r+j} M_{r+k-j}}{M_{r+k}}
=M_r \binom{k+r-1}{r} N_j(k,r)
  \frac{\frac{M_{r+l}}{(r+l)M_{r+l-1}}}
       {\frac{M_{r+k-j+l}}{(r+k-j+l)M_{r+k-j+l-1}}}
\le
  M_r \binom{k+r-1}{r}
\end{align*}
for any $j\le k$. 
\end{proof}

\subsection{Proof of Lemma \ref{lemma-af}}
\begin{proof}
By Lemma \ref{lc-prop} (i) with $l=k-1$, $\{M_l/M_{l-1}\}_{l=1}^\infty$ is strictly increasing. Moreover, by Lemma \ref{lc-prop} (iv) with $l=1$ we have $\lim_{k\to\infty}M_k/M_{k-1}=\infty$. 
Therefore, for any $r\ge 1$ there exists $l\in \N$ such that 
$M_l/M_{l-1}\le r <M_{l+1}/M_l$. 
Then we have
\begin{align*}
  \frac{r^l}{M_l}
    \ge
    \frac{r^k}{M_k} \frac{M_k}{M_l}\(\frac{M_{l}}{M_{l-1}}\)^{l-k}
    \ge
    \frac{r^k}{M_k} \frac{M_k}{M_l}
    \frac{M_{l}}{M_{l-1}}\cdots \frac{M_{k+1}}{M_{k}}
   =\frac{r^k}{M_k} 
\end{align*}
for any $1\le k<l$, and 
\begin{align*}
  \frac{r^l}{M_l}
    \ge
    \frac{r^k}{M_k} \frac{\frac{M_k}{M_l}}{\(\frac{M_{l+1}}{M_l}\)^{k-l}}
    \ge
    \frac{r^k}{M_k} 
    \frac{\frac{M_k}{M_l}}
         {\frac{M_{l+1}}{M_l}\cdots\frac{M_{k}}{M_{k-1}}}
   =\frac{r^k}{M_k} 
\end{align*}
for any $k>l$; thus (\ref{lemma-af-e2}) is proved. 
By the same way, for $0\le r<M_1/M_0$ and $k>1$ we have $r/M_1 \ge r^k/M_k$. 
Moreover, by Lemma \ref{lc-prop} (iii) we have 
\begin{align*}
  \mM\(r;\{M_k\}\)
 \ge \mM\(\frac{M_l}{M_{l-1}};\{M_k\}\)
 =\frac{1}{M_l}\(\frac{M_{l}}{M_{l-1}}\)^l
 \ge \de^{-l}
 \to\infty
 \;\; (l\to\infty). 
\end{align*}
\end{proof}
\subsection{Proof of Lemma \ref{lemma-af1}}
\begin{proof}
We denote $\mM\(r;\{M_k\}\)=\mM(r)$. 
Suppose that $r\in [M_l/M_{l-1},M_{l+1}/M_l)$ and 
$r+d\in [M_m/M_{m-1},M_{m+1}/M_m)$ for $l\le m$. 
By Lemma \ref{lc-prop} (i) and Lemma \ref{lemma-af} we have
\begin{align*}
  \frac{\mM(r+d)}{\mM(r)}
 =&\frac{M_l}{M_m}\(1+\frac{d}{r}\)^l (r+d)^{m-l}
  \le e^d \frac{M_l}{M_m} \(\frac{M_{l+1}}{M_l}+d\)^{m-l}
\\
  \le&2^{m-l-1}e^d  \frac{M_l}{M_m} 
  \(\(\frac{M_{l+1}}{M_l}\)^{m-l}+d^{m-l}\). 
\end{align*}
If $m=l$, then we have $\mM(r+d)/\mM(r)\le e^d$. 
On the other hand, if $m>l$, then noting the inequalities
\begin{align*}
  \frac{M_l}{M_m} \(\frac{M_{l+1}}{M_l}\)^{m-l}
 =\frac{M_{m-1}}{M_m}\cdots\frac{M_l}{M_{l+1}} \(\frac{M_{l+1}}{M_l}\)^{m-l}
  \le 1,
\end{align*}
and
\begin{align*}
  d\ge \frac{M_{m}}{M_{m-1}}-\frac{M_{l+1}}{M_{l}}\ge m-l-1 
\end{align*}
by Lemma \ref{lc-prop} (iv), we have 
\begin{align*}
  \frac{\mM(r+d)}{\mM(r)}
  \le 
  (2e)^{d} \(1 +d^{d+1}\). 
\end{align*}

\end{proof}
\subsection{Proof of Lemma \ref{lemma-af2}}
\begin{proof}
If $M_k/M_{k-1}<N_k/N_{k-1}$, then for any $k\in \N$ we have 
\begin{align*}
  \frac{N_k}{M_k}
 =\frac{\frac{N_k}{N_{k-1}}\cdots\frac{N_1}{N_0}}
       {\frac{M_k}{M_{k-1}}\cdots\frac{M_1}{M_0}}
  >1. 
\end{align*}
Let $r\in[M_m/M_{m-1},M_{m+1}/M_{m})\cap[N_l/N_{l-1},N_{l+1}/N_l)$ 
with $l\le m$. 
Then by Lemma \ref{lemma-af} we have
\begin{align*}
  \frac{\mM(r;\{M_k\})}{\mM(r;\{N_k\})}
=\frac{N_l}{M_m}r^{m-l}
\ge \frac{N_l}{M_m}\(\frac{M_m}{M_{m-1}}\)^{m-l}
=\frac{\(\frac{M_m}{M_{m-1}}\)^{m-l}}
       {\frac{M_{m}}{M_{m-1}}\cdots\frac{M_{l+1}}{M_{l}}}
  \frac{N_l}{M_l}
\ge\frac{N_l}{M_l}.  
\end{align*}
Analogously, we have 
\begin{align*}
   \frac{\mM(r;\{M_k\})}{r\:\mM(r;\{N_k\})}
  =\frac{N_l}{M_{l+1}}\frac{r^{m-l-1}}{\frac{M_m}{M_{l+1}}}
\ge \frac{N_l}{M_{l+1}}
    \frac{\(\frac{M_m}{M_{m-1}}\)^{m-l-1}}
         {\frac{M_m}{M_{m-1}}\cdots\frac{M_{l+2}}{M_{l+1}}}
\ge \frac{N_l}{M_{l+1}}
\end{align*}
for $m\ge l+1$, and 
\begin{align*}
   \frac{\mM(r;\{M_k\})}{r\:\mM(r;\{N_k\})}
\ge\frac{N_l}{M_{l+1}}\frac{r^{-1}}{\frac{M_l}{M_{l+1}}}
\ge\frac{N_l}{M_{l+1}}
\end{align*}
for $l=m$. 
Thus the proof of (\ref{lemma-af2-e1}), (\ref{lemma-af2-e2}) and (\ref{lemma-af2-e3}) are concluded. 
\end{proof}

\subsection{Proof of Proposition \ref{Prop-cB}}
(i) and (ii) are trivial by the definition and Lemma \ref{lemma-af2}. 
We denote 
\begin{equation*}
  \cM(r)=\mM(r;\{M_k\}), 
  \;\;
  \tcM(r)=\mM(r;\{\tM_k\}), 
  \;\;
  \cN(r)=\mM(r;\{N_k\}), 
  \;\;
  \tcN(r)=\mM(r;\{\tN_k\}). 
\end{equation*}
Let us define the sequences of positive real numbers $\{r_j\}_{j=1}^\infty$ and $\{\rho_j\}_{j=1}^\infty$ by 
\begin{equation*}
  r_1=\rho_1=1,\;\;
  \rho_{j+1}=r_j+1,\;\;
  r_{j+1}=\cM\(\frac{r_{j+1}}{\rho_{j+1}}\)
\end{equation*}
for $j\ge 1$. 
Then we have the following lemmas: 
\begin{Lemm}\label{Prop-cB-lemm1}
$\{\rho_j\}$, $\{r_j\}$ and $\{r_j/\rho_j\}$ are strictly increasing 
and go to infinity as $j\to\infty$. 
\end{Lemm}
\begin{proof}
Let $j\ge 1$. 
If $r_{j+1}/\rho_{j+1}< M_2/M_1$, then we have 
\begin{align*}
  r_{j+1}=\cM\(\frac{r_{j+1}}{\rho_{j+1}}\)
  =\frac{1}{M_1}\frac{r_{j+1}}{\rho_{j+1}}=\frac{r_{j+1}}{r_{j}+1},
\end{align*}
and thus $r_{j+1}=0$. 
Therefore, it must be that $r_{j+1}/\rho_{j+1}\ge M_2/M_1$ for any $j\ge 1$. 
Let $r_{j+1}/\rho_{j+1}\in [M_k/M_{k-1},M_{k+1}/M_{k})$ for $k\ge 2$. 
Then by Lemma \ref{lemma-af} we have 
\begin{align*}
  r_{j+1}=\cM\(\frac{r_{j+1}}{\rho_{j+1}}\)
   =\frac{1}{M_k}\(\frac{r_{j+1}}{\rho_{j+1}}\)^k,
\end{align*}
it follows that
\begin{align*}
  r_{j+1}=M_k^{\frac{1}{k-1}}(r_j+1)^{\frac{k}{k-1}}
  >(r_j+1)^{\frac{k}{k-1}}>r_j+1.
\end{align*}
Therefore, $\{r_j\}$, $\{\rho_j\}$ and $\{\cM(r_j/\rho_j)\}$ are strictly increasing and all the sequences go to infinity as $j\to\infty$. 
Moreover, by Lemma \ref{lemma-af} we see that $\{r_j/\rho_j\}$ is strictly increasing and $\lim_{j\to\infty}r_j/\rho_j=\infty$. 
\end{proof}
%
\begin{Lemm}\label{Prop-cB-lemm2}
For any $j\ge 2$ the following inequality is valid$:$ 
\begin{equation}\label{Prop-cB-lemm2-e1}
  \frac{r_j}{\rho_j-1} \le \frac{r_j}{\rho_j} + 2M_2. 
\end{equation}
\end{Lemm}
\begin{proof}
Noting $2\le r_1 +1 =\rho_2 \le \rho_j$ for any $j\ge 2$, we have 
\begin{align*}
  \frac{1}{\rho_{j}-1}\le \frac{1}{\rho_j}+\frac{2}{\rho_j^2}. 
\end{align*}
Therefore, by Lemma \ref{lemma-af} we have 
\begin{equation*}
  \frac{r_j}{\rho_j-1}
 \le \frac{r_j}{\rho_j}+\frac{2r_j}{\rho_j^2}
 =\frac{r_j}{\rho_j} + 2M_2 \frac{\frac{1}{M_2}{\(\frac{r_j}{\rho_j}\)^2}}{r_j}
 \le \frac{r_j}{\rho_j} + 2M_2 \frac{\cM\(\frac{r_j}{\rho_j}\)}{r_j}
 =\frac{r_j}{\rho_j} + 2M_2. 
\end{equation*}
\end{proof}
We define $\hat{u}_{1,j}=\hat{u}_{1,j}(\xi)$ for $j=1,2,\cdots$ by
\begin{equation*}
  \hat{u}_{1,j}(\xi):=
  \begin{cases}
    \(\om_n r_j^{n-2} \tcM\(\frac{|\xi|}{\rho_j}\)\)^{-\frac12}
      & |\xi|\in[r_j,r_{j}+1),\\
    0 & |\xi|\not\in[r_j,r_j+1),
  \end{cases}
\end{equation*}
where $\om_n=\pi^{n-1}\sum_{l=0}^{n-2}\binom{n-1}{l}$. 
Then we have 
\begin{align*}
G(\rho_j,\eta,\{M_k\},\{\tM_k\};0,u_{1,j})
=&\frac12\int_{|\xi|\ge \rho_j}\tcM\(\frac{|\xi|}{\rho_j}\)
    \exp\(\frac{\eta|\xi|}{\cM\(\frac{|\xi|}{\rho_j}\)}\)
    |\hat{u}_{1,j}(\xi)|^2\,d\xi
\\
 &=\frac{1}{2\om_n r_j^{n-2}}
  \int_{r_j\le |\xi|\le r_j+1}
    \exp\(\frac{\eta|\xi|}{\cM\(\frac{|\xi|}{\rho_j}\)}\)
    \,d\xi
\\
 &\le\frac{1}{2\om_n r_j^{n-2}}
  \exp\(\frac{\eta r_j}{\cM\(\frac{r_j}{\rho_j}\)}\)
  \int_{r_j\le |\xi|\le r_j+1}\,d\xi
\\
 &\le\frac{\pi^{n-1}e^\eta}{\om_n r_j^{n-2}}\((r_j+1)^{n-1}-r_j^{n-1}\)
 =\frac{\pi^{n-1}e^\eta}{\om_n r_j^{n-2}}
  \sum_{l=0}^{n-2}\binom{n-1}{l}r_j^l
\\
  &\le
  \frac{\pi^{n-1}e^\eta}{\om_n}\sum_{l=0}^{n-2}\binom{n-1}{l}
 =e^\eta. 
\end{align*}
Thus the sequence of initial data $\{(0,u_{1,j})\}_{j=1}^\infty$ is bounded in $\cB_\Delta(\{M_k\},\{\tM_k\})$. 

\noindent (iii): 
For any $\rho>0$, there exists $j\in \N$ such that $\rho_j\ge \rho$. 
Then for large $j$ we have 
\begin{align*}
  F_j:=&
  \int_{\R^n}\mM\(\frac{|\xi|}{\rho};\{L_k\}\)|\hat{u}_{1,j}(\xi)|^2\,d\xi
\ge
  \frac{1}{\om_n r_j^{n-2}}
  \int_{r_j\le|\xi|\le r_j+1}
  \frac{\mM\(\frac{|\xi|}{\rho_j};  \{L_k\}\)}
       {\mM\(\frac{|\xi|}{\rho_j};\{\tM_k\}\)}
  \,d\xi
\\
\ge&
  \frac{2\pi^{n-1}(n-1)}{C\om_n}
  \frac{\mM\(\frac{r_j}{\rho_j};\{L_k\}\)}
       {\mM\(\frac{r_j}{\rho_j};\{\tM_k\}\)}. 
\end{align*}
Therefore, by Lemma \ref{lemma-af1}, \ref{lemma-af2} and \ref{Prop-cB-lemm1} we have $\lim_{j\to\infty} F_j=\infty$ ($j\to\infty$), it follows that 
$\{u_{1,j}\}$ is not bounded in $\cH(\{L_k\})$; 
thus (iii) is proved. 

\noindent
(iv) Let us prove that 
$\sup_{l}\{G(\si_l,\eta,\{N_k\},\{\tN_k\};0,u_{1,j})\}$ is not bounded as $j\to\infty$ for any $\{\si_l\}\in {\cal L}$ and $\eta>0$. 
We note that $M_k/M_{k-1}<N_k/N_{k-1}$ is valid if $N_k/M_{k+1}\nearrow\infty$ as $k\to\infty$. 
By Lemma \ref{Prop-cB-lemm1}, for any large $l\in \N$ there exists $j\ge 2$ such that $\si_l\in[r_{j-1},r_j)$. 
Then for any $r_j\le |\xi|\le r_j+1$ we have 
\begin{align*}
  \tcN\(\frac{|\xi|}{\si_l}\)
  \ge
  \tcN\(\frac{r_j}{\si_l}\)
  \ge
  1 
\end{align*}
and 
\begin{align*}
  \tcM\(\frac{|\xi|}{\rho_j}\)
  \le
  \tcM\(\frac{r_j+1}{\rho_j}\)
  \le
  \tcM\(\frac{r_j}{\rho_j}+1\)
  \le C \tcM\(\frac{r_j}{\rho_j}\)
\end{align*}
for a constant $C>1$ by Lemma \ref{lemma-af1}. 
Moreover, if $\si_l\in[r_{j-1},r_{j-1}+1)=[\rho_j-1,\rho_j)$, 
then by Lemma \ref{lemma-af1}, (\ref{Prop-cB-lemm2-e1}) 
and noting that $r/\cN(r)$ is monotone decreasing, 
there exists a positive constant $\ve$ such that 
\begin{align*}
  \frac{|\xi|}{\tcN\(\frac{|\xi|}{\si_l}\)}
  \ge& \frac{r_j+1}{\cN\(\frac{r_j+1}{\si_l}\)}
  \ge \frac{r_j}{\cN\(\frac{r_j+1}{\si_l}\)}
  \ge \frac{r_j}{\cN\(\frac{r_j+1}{\rho_j-1}\)}
  \ge \frac{r_j}{\cN\(\frac{r_j}{\rho_j-1}+\frac{1}{\rho_2-1}\)}
\\
  \ge& \frac{r_j}{\cN\(\frac{r_j}{\rho_j}+2M_2+1\)}
  \ge \ve\frac{r_j}{\cN\(\frac{r_j}{\rho_j}\)} 
  = \ve\frac{\cM\(\frac{r_j}{\rho_j}\)}{\cN\(\frac{r_j}{\rho_j}\)}. 
\end{align*}
On the other hand, if $\si_l\in[r_{j-1}+1,r_j)=[\rho_j,\rho_{j+1}-1)$, then for $r_j\le |\xi| \le r_j+1$ we have 
\begin{align*}
  \frac{|\xi|}{\cN\(\frac{|\xi|}{\si_l}\)}
  \ge \frac{|\xi|}{\cN\(\frac{|\xi|}{\rho_j}\)}
  \ge \frac{r_j+1}{\cN\(\frac{r_j+1}{\rho_j}\)}
  \ge \frac{r_j}{\cN\(\frac{r_j}{\rho_j}+\frac{1}{2}\)}
  \ge \ve\frac{\cM\(\frac{r_j}{\rho_j}\)}{\cN\(\frac{r_j}{\rho_j}\)}. 
\end{align*}
Summarizing the estimates above and noting Remark \ref{rem-af-k!}, there exists a positive constant $\tilde{\eta}$ such that 
\begin{align*}
  G(\si_l,\eta,\{N_k\},\{\tN_k\};0,u_{1,j})
  =&\frac12\int_{|\xi|\ge \si_l}\tcN\(\frac{|\xi|}{\si_l}\)
    \exp\(\frac{\eta|\xi|}{\cN\(\frac{|\xi|}{\si_l}\)}\)
   |\hat{u}_{1,j}(\xi)|^2\,d\xi
\\
  =&\frac{1}{2\om_n r_j^{n-2}}\int_{r_j\le |\xi|\le r_j+1}
   \frac{\tcN\(\frac{|\xi|}{\si_l}\)}{\tcM\(\frac{|\xi|}{\rho_j}\)}
   \exp\(\frac{\eta|\xi|}{\cN\(\frac{|\xi|}{\si_l}\)}\)
   \,d\xi
\\
  \ge&
  \frac{1}{2C\om_n r_j^{n-2}}
   \frac{1}{\tcM\(\frac{r_j}{\rho_j}\)}
   \exp\(\ve\eta\frac{\cM\(\frac{r_j}{\rho_j}\)}
                     {\cN\(\frac{r_j}{\rho_j}\)}\)
  \int_{r_j\le |\xi|\le r_j+1}
   \,d\xi
\\
  \ge&
  \frac{\pi^{n-1}}{C\om_n r_j^{n-2}}
   \exp\(\ve\eta\frac{\cM\(\frac{r_j}{\rho_j}\)}
                     {\cN\(\frac{r_j}{\rho_j}\)}-\tilde{\eta}\frac{r_j}{\rho_j}\)
   \sum_{k=0}^{n-2}\binom{n-1}{k}r_j^k
\\
  \ge&
  \frac{\pi^{n-1}(n-1)}{C\om_n}
   \exp\(\ve\eta\frac{\cM\(\frac{r_j}{\rho_j}\)}
                     {\cN\(\frac{r_j}{\rho_j}\)}-\tilde{\eta}\frac{r_j}{\rho_j}\).
\end{align*}
Then Lemma \ref{lemma-af2} and \ref{Prop-cB-lemm1} imply that 
$\lim_{j\to\infty}\sup_l\{G(\si_l,\eta,\{N_k\},\{\tN_k\};0,u_{1,j})\}=\infty$. 
Consequently, the sequence of initial data $\{(0,u_{1,j})\}$ is not bounded in $\cB_\Delta(\{N_k\},\{\tN_k\})$.

\subsection{Proof of the inequality (\ref{sum_k^-2})}
We suppose that $r_1\le r_2$ without loss of generality. 
Noting the inequalities 
$(k+r_2-r_1-1)/2 \le j_0:=[(k+r_2-r_1)/2] \le (k+r_2-r_1)/2$, 
we have 
\begin{align*}
  &\sum_{j=0}^{j_0} \(\frac{r_1+r_2+k+j+1}{(r_1+j+1)(r_2+k-j+1)}\)^2
  \\
  &\le
    \(\frac{r_1+r_2+k+j_0+1}{r_2+k-j_0+1}\)^2
    \sum_{j=0}^{j_0} \frac{1}{(r_1+j+1)^2}
   \le
    \(\frac{r_1+3r_2+3k+3}
           {r_1+r_2+k+1}\)^2
    \sum_{j=0}^{\infty} \frac{1}{(j+1)^2}
  \\
   &
  \le \frac{3\pi^2}{2}
\end{align*}
and
\begin{align*}
  &\sum_{j=j_0+1}^{k} \(\frac{r_1+r_2+k+j}{(r_1+j+1)(r_2+k-j+1)}\)^2
  \\
  &\le \(\frac{r_1+r_2+k+j_0+2}{r_1+j_0+2}\)^2
   \sum_{j=j_0+1}^{k} \frac{1}{(r_2+k-j+1)^2}
   \le \(\frac{r_1+3r_2+3k+3}{r_1+r_2+k+3}\)^2
   \sum_{l=0}^{\infty} \frac{1}{(l+1)^2}
  \\
   &
  \le \frac{3\pi^2}{2} 
\end{align*}
for any $k\in \N$.

\begin{Lemm}\label{Lemm_sqrt_f}
Let $\{M_k\}$ satisfy (\ref{lc}) and $f\in C^\infty([0,\infty))$ satisfy 
$f(t)\ge 1$. 
If there exists a positive constant $\nu$ such that 
\begin{equation*}
  \left|f^{(k)}(t)\right|
  \le \nu^k M_k
\end{equation*}
for any $k\in \N$, 
then the estimates 
\begin{equation}\label{Lemm_sqrt_f-e2}
  \left|\frac{d^k}{dt^k}\sqrt{f(t)}\right|
  \le \mu^k M_k
\end{equation}
are valid for any $k\in \N$ with $\mu=16\pi^2\nu/9$. 
\end{Lemm}
\begin{proof}
Denoting $\ve=3/(4\pi^2)(<1)$ and setting $\mu=4\nu/\ve$, we have 
\begin{align*}
  \left|f^{(k)}(t)\right|
  \le \ve\frac{M_k \mu^k}{(k+1)^2} 
     \frac{(k+1)^2}{4^k}
  \le \ve\frac{M_k \mu^k}{(k+1)^2}
\end{align*}
for any $k\in \N$. 
We shall show that the estimate 
\begin{equation}\label{Lemm_sqrt_f-e22}
  \left|\frac{d^k}{dt^k}\sqrt{f(t)}\right|
  \le \ve \frac{M_k \mu^k}{(k+1)^2}
\end{equation}
is valid for any $k\ge 1$ by induction. 
Indeed, we immediately have (\ref{Lemm_sqrt_f-e2}) from (\ref{Lemm_sqrt_f-e22}). 
For $k=1$ (\ref{Lemm_sqrt_f-e22}) is valid as follows: 
\begin{align*}
  \left|\frac{d}{dt}\sqrt{f(t)}\right|
 =\frac{|f'(t)|}{2\sqrt{f(t)}}
 \le \frac{\ve}{2} \frac{M_1 \mu}{2^2}. 
\end{align*}
Let us suppose that (\ref{Lemm_sqrt_f-e22}) is valid for $k=1,\ldots,j$. 
Denoting $g(t)=\sqrt{f(t)}$, we have 
\begin{equation*}
  g^{(j+1)}(t)
 =\frac{1}{2g(t)}f^{(j+1)}(t)
 -\frac{1}{2g(t)}\sum_{k=1}^{j}\binom{j+1}{k}g^{(k)}(t)g^{(j+1-k)}(t). 
\end{equation*}
Therefore, by using Lemma \ref{lc-binomial} and (\ref{sum_k^-2}), we obtain 
\begin{align*}
  \left|g^{(j+1)}(t)\right|
  \le&
    \frac{\ve}{2} \frac{M_{j+1} \mu^{j+1}}{(j+2)^2}
   +\frac{\ve^2}{2}\frac{M_{j+1}\mu^{j+1}}{(j+2)^2}
    \sum_{k=1}^{j}\binom{j+1}{k}
    \frac{M_k M_{j+1-k}}{M_{j+1}}
    \frac{(j+2)^2}{(k+1)^2(j-k+2)^2}
\\
  \le&
    \ve \frac{M_{j+1} \mu^{j+1}}{(j+2)^2},
\end{align*}
Thus (\ref{Lemm_sqrt_f-e22}) is valid for any $k\ge 1$. 
\end{proof}

\section*{Acknowledgment}
The author is supported by JSPS Grant-in-Aid for Scientific Research (C) 26400170.


\begin{thebibliography}{99}
\bibitem{AS} A. Arosio and S. Spagnolo, 
Global solutions of the Cauchy problem for a nonlinear hyperbolic equation, in: H. Brezis, J. L. Lions (Eds.), Nonlinear PDEfs and their applications, Coll\`ege de France Seminar, vol. VI, Research Notes in Mathematics, vol. 109, Pitman, Boston, 1984, pp. 1--26. 

\bibitem{B}S. Bernstein,
Sur une class d'\'equations fonctionnelles aux d\'eriv\'ees partielles,
Izvestia. Akad. Nauk SSSR 4 (1940) 17--26.


\bibitem{DS} P. D'Ancona and S. Spagnolo, 
A class of nonlinear hyperbolic problems with global solutions, 
Arch. Rat. Mech. Anal. 124 (1993), 201--219.


\bibitem{GG} M. Ghisi and M. Gobbino, 
Kirchhoff equations from quasi-analytic to spectral-gap data. Bull. Lond. Math. Soc. 43 (2011), 374--385.


\bibitem{GH}J. M. Greenberg and S. C. Hu,
The initial value problem for a stretched string,
Quarterly of Applied Mathematics 38 (1980/81) 289--311. 


\bibitem{H98}F. Hirosawa,
Global solvability of the degenerate Kirchhoff equation, Ann. Scuola Norm. Sup. Pisa Cl. Sci. 26 (4) (1998) 75--95.


\bibitem{H06}F. Hirosawa,
Global solvability for Kirchhoff equation in special classes of non-analytic functions.
J. Differential Equations 230 (2006), 49--70.


\bibitem{HI13}F. Hirosawa and H. Ishida,
On second order weakly hyperbolic equations and the ultradifferentiable classes.J. Differential Equations 255 (2013), 1437--1468.

\bibitem{Kr} G. Krantz and H.R. Parks, 
A primer of real analytic functions, Basler Lehrbucher, Vol. 4, BirkhNauser
Verlag, Basel, 1992.

\bibitem{K} G. Kirchhoff, Forlesungen \"uber Mechanik,'' Teubner, Leipzig, 1883.
\bibitem{M02}R. Manfrin,
Global solvability to the Kirchhoff equation for a new class of initial data,
Port. Math. $($N.S.$)$, 59 (2002), 91--109.


\bibitem{M05}R. Manfrin,
On the global solvability of Kirchhoff equation for non-analytic initial data,
J. Differential Equations, 211 (2005), 38--60. 


\bibitem{Me}G. P. Menzala,
On classical solutions of a quasilinear hyperbolic equation, 
Nonlinear Analysis 3 (1979), 613--627.



\bibitem{Ni}K. Nishihara,
On a global solution of some quasilinear hyperbolic equation, 
Tokyo J. Math. 7 (1984), 437--459.


\bibitem{Po} S.I. Pohozaev, On a class of quasilinear hyperbolic equations, Math. USSR Sb. 25 (1975) 145--158.


\bibitem{S} S. Spagnolo, The Cauchy problem for Kirchhoff equations. Proceedings of the Second International Conference on Partial Differential Equations. 
Rend. Sem. Mat. Fis. Milano 62 (1992), 17--51 (1994). 


\bibitem{Yg} K. Yagdjian, Geometric optics for the Kirchhoff-type equations. 
J. Anal. Math. 89 (2003), 57--112. 


\bibitem{Yz} T. Yamazaki, Global solvability for the Kirchhoff equations in
exterior domains of dimension three, 
J. Differential Equations, 210 (2005), 290--316. 

\end{thebibliography}
\end{document}